\documentclass[11pt, reqno]{amsart}
\usepackage{amssymb}
\usepackage{amsmath}
\usepackage{enumitem}
\usepackage{hyperref}
\usepackage[normalem]{ulem}
\usepackage{color}
\newtheorem{Theorem}{Theorem}[section]
\newtheorem{Lemma}[Theorem]{Lemma}
\newtheorem{Proposition}[Theorem]{Proposition}
\newtheorem{Corollary}[Theorem]{Corollary}
\newtheorem{Remark}[Theorem]{Remark}

\newcommand{\R}{\mathbb R}
\newcommand{\eps}{\varepsilon}
\newcommand{\dsigma}{d\sigma}

\newcommand{\nero }{\color{black}}

\begin{document}
 \title[Weighted Hardy-Rellich inequalities]{Weighted Hardy-Rellich inequalities\\ via the Emden-Fowler transform}
\author{Elvise Berchio}
\address{\hbox{\parbox{5.7in}{\medskip\noindent{Dipartimento di Scienze Matematiche, \\
Politecnico di Torino,\\
        Corso Duca degli Abruzzi 24, 10129 Torino, Italy. \\[3pt]
        \em{E-mail address: }{\tt elvise.berchio@polito.it}}}}}

\author{Paolo Caldiroli}
\address{\hbox{\parbox{5.7in}{\medskip\noindent{Dipartimento di Matematica ``Giuseppe Peano'', \\
Universit\`a degli Studi di Torino,\\
   Via Carlo Alberto 10, 10123 Torino, Italy. \\[3pt]
        \em{E-mail address: }{\tt paolo.caldiroli@unito.it}}}}}


\begin{abstract}
We exploit a technique based on the Emden-Fowler transform to prove optimal Hardy-Rellich inequalities on cones, including the punctured space $\R^N\setminus\{0\}$ and the half space as particular cases. We find optimal constants for classes of test functions vanishing on the boundary of the cone and possibly orthogonal to prescribed eigenspaces of the Laplace Beltrami operator restricted to the spherical projection of the cone. Furthermore, we show that extremals do not exist in the natural function spaces. Depending on the parameters, certain resonance phenomena can occur. For proper cones, this is excluded when considering test functions with compact support. 
Finally, for suitable subsets of the cones we provide improved Hardy-Rellich inequalities, under different boundary conditions, with optimal remainder terms.
\end{abstract}

\maketitle

\normalsize
\section{Introduction}
The classical Rellich inequality \cite{Rellich} states that
\[
\int_{\R^N}|\Delta u|^2\,dx\ge \left[\frac{N(N-4)}4\right]^2\int_{\R^N}|x|^{-4}|u|^2\,dx\quad\forall u\in C^2_c(\R^N\setminus\{0\}),~~N\ne 2\,.
\]
This inequality has inspired extensive research, leading to numerous generalizations and refinements. Among them, the following inequality has recently gained particular interest: 
\begin{equation}\label{HardyRellich}
\int_{\R^N}|x|^\alpha|\Delta u|^2\,dx\ge C_{N,\alpha}\int_{\R^N}|x|^{\alpha-2}|\nabla u|^2\,dx\quad\forall u\in C^2_c(\R^N\setminus\{0\})\,,
\end{equation}
for some $C_{N,\alpha} \geq 0$. We refer to this as power weighted Hardy-Rellich inequality, because, in some sense, it lies between the power weighted version of the Hardy inequality
\begin{equation}
\label{Hardy}
\int_{\R^N}|x|^\alpha|\nabla u|^2\,dx\ge \left(\frac{N-2+\alpha}2\right)^2\int_{\R^N}|x|^{\alpha-2}|u|^2\,dx\quad\forall u\in C^1_c(\R^N\setminus\{0\})\,,
\end{equation}
and the corresponding version of the Rellich inequality
\begin{equation}
\label{Rellich}
\int_{\R^N}|x|^\alpha|\Delta u|^2\,dx\ge c_{N,\alpha}\int_{\R^N}|x|^{\alpha-4}|u|^2\,dx\quad\forall u\in C^2_c(\R^N\setminus\{0\})\,,
\end{equation}
whose optimal constant is
\[
c_{N,\alpha}=\min_{j=0,1,2,...}\left(\gamma_\alpha+\lambda_j\right)^2,
\]
where
\begin{equation}\label{gamma-alpha}
\gamma_\alpha=\left(\frac{N-2}2\right)^2-\left(\frac{\alpha-2}2\right)^2\quad\text{and}\quad\lambda_j=j(j+N-2)\quad\forall j=0,1,2,...
\end{equation}
(see \cite{CM} and \cite{GM} and references therein). 
Inequality \eqref{HardyRellich} has a more recent history, compared to \eqref{Hardy}--\eqref{Rellich}. It first appeared in the work \cite{TZ} by Tertikas and Zographopoulos, who established it for $\alpha=0$ in higher dimensions $N\ge 5$, with best constant $C_{N,0}=\frac{N^2}4$ for all $N\ge 5$. Subsequently, Beckner \cite{Beckner} and Ghoussoub and Moradifam \cite{GM}, independently and with different techniques, extended it to all dimensions $N\ge 3$. For low dimensions the optimal constants are $C_{4,0}=3$ and $C_{3,0}=\frac{25}{36}$. We point out that inequality \eqref{HardyRellich} for $\alpha=0$ becomes trivial in dimension $N=2$, i.e. $C_{2,0}=0$, whereas in dimension $N=1$ the best constant is $C_{1,0}=\frac14$. A unified approach for deriving the optimal constant in \eqref{HardyRellich} with $\alpha=0$, valid for all dimensions, was recently provided in \cite{Cazacu}.

As far as concern the power weighted case, the optimal constant in  \eqref{HardyRellich} turns out to be
\begin{equation}\label{best-constant}
C_{N,\alpha}=\min\left\{\left(\frac{N-\alpha}2 \right)^2\right\}\cup\left\{\frac{\left(\gamma_\alpha+\lambda_j\right)^2}{\left(\frac{N-4+\alpha}2\right)^2+\lambda_j}~:~j=1,2,3,...\right\}
\end{equation}
where $\gamma_\alpha$ and $\lambda_j$ are as in \eqref{gamma-alpha}. 
This result was initially obtained under parameter restrictions in \cite{GM} (see also \cite{CM2}) and later generalized in  \cite{CCF, GPPS}. Further extensions to vector fields were studied in  \cite{HT-MathAnn, HT-JFA, HT-2021}.

Various approaches and tools have been employed to study \eqref{Rellich}, such as Bessel pairs and spherical harmonics decomposition (\cite{GM}), Fourier transform tools (\cite{Beckner}), factorization of differential operators (\cite{GPPS}), etc. In this paper, we adopt a method previously used in \cite{CM} and \cite{CM2} based on the Emden-Fowler transform, to deal with \eqref{Rellich} on cones, i.e., domains of the form 
\begin{equation*}
C_\Sigma:=\left\{x\in\R^N\setminus\{0\}~:~\frac x{|x|}\in\Sigma\right\},
\end{equation*}
where $\Sigma$ is a $C^2$ domain in the sphere $S^{N-1}$ ($N\geq 2$). We remark that when $\Sigma=S^{N-1}$, then $C_\Sigma$ is the punctured space $\R^N\setminus\{0\}$ while when $\Sigma$ is a half-sphere, the associated cone becomes a half-space, which is a smooth domain with a boundary that includes the singularity.

The Emden-Fowler transform converts problems on the cone $C_\Sigma$ into problems on the cylinder $\R\times\Sigma$, it changes power weighted $L^2$-norms of a mapping and of its derivatives into corresponding weight-free $L^2$-norms with constant coefficients, and it decouples the problem into a 1-dimensional problem and eigenvalues problems for the Laplace-Beltrami operator with null boundary conditions on $\Sigma$.
This is accomplished in Section \ref{S2}, where we determine the optimal constant in the Hardy-Rellich inequality for test functions $u\in C^2(\overline{C_\Sigma})$ such that $u|_{\partial C_\Sigma}=0$ and vanishing in neighborhoods of $0$ and $\infty$ (see Theorem \ref{HR-cone}). Our approach provides optimal constants also when restricting to subclasses of test functions that are orthogonal to prescribed eigenspaces of the Laplace-Beltrami operator, and we show that extremals do not exist in the natural function space (Theorem \ref{not_att}).  

We point out that, while the Rellich inequality \eqref{Rellich} is invariant under the Kelvin transform
\[
u\mapsto\widehat{u}\quad\text{where}\quad\widehat{u}(x):=|x|^{2-N}u\left(\frac{x}{|x|}\right)
\] 
(except than $\alpha$ becomes $4-\alpha$), the Hardy-Rellich inequality \eqref{HardyRellich} loses this invariance for $\alpha\ne 2$. Instead, an optimal Schmincke-type inequality emerges (Theorem \ref{T:Schmincke-type}). 

For any cone $C_\Sigma$, as in the case of the whole punctured space, the optimal Hardy-Rellich constant may vanish for certain values of the parameters. This occurs due to some resonance phenomena, when $-\gamma_\alpha$ is an eigenvalue of the  Laplace-Beltrami operator in $H^1_0(\Sigma)$; see, for instance, \eqref{best-constant} when $\Sigma=S^{N-1}$. In Section \ref{S3} we prove that for \emph{proper} cones, if we restrict to test functions $u\in C_c^2({C_\Sigma})$, such resonance phenomena do not occur, and the optimal constant is always positive. 

Finally, in Section \ref{S5}, we exploit the effectiveness of the Emden-Fowler trasform method to deal with a class of cone-like domains of the form 
\begin{equation}
\label{general-domains}
\Omega=\{x\in C_\Sigma~:~|x|<1\}\quad\text{or}\quad\Omega=\{x\in C_\Sigma~:~|x|>1\}\,.
\end{equation}
In this case, according to the boundary conditions, we show some new improved Hardy Rellich inequalities with remainder terms involving Leray-type weights, which also contain a logarithmic part (for integral inequalities with first order differential operators see the very first work \cite{L} and, more recently, \cite{Mu} and the references therein).  Furthermore, optimal improvements are also provided when imposing the above mentioned orthogonality conditions.	

To the best of our knowledge, prior results in this direction were obtained in the so-called Dirichlet case, in \cite{AGS},  \cite{BT} and \cite{TZ}, for test functions supported in bounded domains containing the origin, under suitable restrictions on the parameters, and in the recent work \cite{GPS} for test functions supported in the punctured ball, see Remark \ref{Dirichlet-rem_terms}. Still in the ball, optimal improved Hardy-Rellich inequalities in the Navier case were instead obtained in  \cite{M}. In the present paper, we also consider mixed Navier-Dirichlet boundary conditions, \nero for domains like \eqref{general-domains}, and we address the optimality of the coefficients in front of the remainder terms as well. 

An example of the type of improved inequalities we obtained is the following  statement, which is a special case of Theorems \ref{ND-case} and \ref{N-case}, for $\alpha=0$ and $\Sigma$ being a half-sphere, see also Remark \ref{R:principal}.

\begin{Proposition}
Let $B_+=\{x\in\R^N~:~x_1>0\,,|x|<1\}$. Then:
\begin{itemize}[leftmargin=15pt]
\item[1.] (\emph{Navier-Dirichlet case}) For every $u\in C^2(\overline{B_+})$ with $u|_{\partial B_+}=0$, $u=0$ in neighborhoods of $0$ and of $\overline{B_+}\cap\{|x|=1\}$ one has
\begin{equation}\label{intro1}
\begin{split}
\int_{B^+}&|\Delta u|^2\,dx\ge\frac{(N^2-4)^2}{4[(N-2)^2+8]}\int_{B^+}|x|^{-2}|\nabla u|^2\,dx\\
&+\frac14\int_{B^+}|x|^{-2}\left|\log|x|\right|^{-2}|\nabla u|^2\,dx+\frac{4(N-1)}{(N-2)^2+8}\int_{B^+}|x|^{-4}\left|\log|x|\right|^{-2}|u|^2\,dx\,.
\end{split}
\end{equation}
\item[2.] (\emph{Navier case}) For every $u\in C^2(\overline{B_+})$ with $u|_{\partial B_+}=0$ and $u=0$ in a neighborhood of $0$, one has
\begin{equation}\label{intro2}
\begin{split}
\int_{B^+}|\Delta u|^2\,dx&\ge\frac{(N^2-4)^2}{4[(N-2)^2+8]} \int_{B^+}|x|^{-2}|\nabla u|^2\,dx\\
&\quad+\frac{(N-2)^4+16(N^2+4)}{16[(N-2)^2+8]}\int_{B^+}|x|^{-4}\left|\log|x|\right|^{-2}|u|^2\,dx\,.
\end{split}
\end{equation}
\end{itemize}
Moreover, both in  \eqref{intro1} and \eqref{intro2}, all coefficients are optimal. 
\end{Proposition}

We refer to Section \ref{S5} for a more complete exposition of our results for cone-like domains like \eqref{general-domains}, in the general weighted case and under the above mentioned orthogonality conditions. In the remarks in that section, we also compare our results with those already known, found in \cite{AGS, GPS, M}. 

\subsection*{Notation}
For reader's convenience, we present some of the notation used throughout this work.
\begin{itemize}[leftmargin=*]
\item $\Sigma$ denotes a $C^2$ domain in $S^{N-1}$ (with $N\ge 2$), 
\begin{equation}\label{CSigma}
C_\Sigma:=\left\{x\in\R^N\setminus\{0\}~:~\frac x{|x|}\in\Sigma\right\}
\end{equation}
is the corresponding cone in $\R^N$, and
\begin{equation}\label{ZSigma}
Z_\Sigma:=\{(t, \omega) \in \R\times\Sigma\}
\end{equation} is the corresponding cylinder in $\R^N$.\vspace{6pt}
\item
$\Lambda_\Sigma$ denotes the spectrum of the Laplace-Beltrami operator in $H^1_0(\Sigma)$ (which is a countable discrete set in $[0,\infty)$) while $\lambda_{\Sigma}$ denotes the principal (i.e. least) eigenvalue. In particular, when $\Sigma=S^{N-1}$, $\Lambda_\Sigma$ is the spectrum of the Laplace-Beltrami operator in $H^1(S^{N-1})$, given by the set $\{j(j+N-2)~:~j=0,1,2,...\}$ and $\lambda_{S^{N-1}}=0$.\vspace{6pt}
\item
Given a cone $C_\Sigma$, we introduce the spaces
\begin{equation}\begin{split}
\label{C20}
C^2_c(C_\Sigma)&:=\{u\in C^2(C_\Sigma)~:~\mathrm{supp}(u)\subset C_\Sigma\}\\
C^2_0(C_\Sigma)&:=\{u\in C^2(\overline{C_\Sigma})~:~u|_{\partial C_\Sigma}=0,~u=0\text{ in neighborhoods of $0$ and $\infty$}\}.
\end{split}
\end{equation} 
Note that when $\Sigma=S^{N-1}$, then $C_\Sigma=\R^N\setminus\{0\}$, and  $C^2_0(C_\Sigma)=C^2_c(C_\Sigma)=C^2_c(\R^N\setminus\{0\})$. 
\vspace{6pt}
\item
Given a cylinder $Z_\Sigma$, we introduce the spaces
\begin{equation}\begin{split}
\label{C2Z}
C^2_c(Z_\Sigma)&:=\{w\in C^2(\overline{Z_\Sigma})~:~\mathrm{supp}(w)\subset Z_\Sigma\}\\
C^2_0(Z_\Sigma)&:=\{w\in C^2(\overline{Z_\Sigma})~:~w|_{\partial Z_\Sigma}=0\}\,.
\end{split}
\end{equation} 
\vspace{6pt}
\item 
For $\alpha\in\R$ we denote
\begin{equation}
\label{ga}
\begin{split}
\gamma_\alpha:=\left(\frac{N-2}2\right)^{\!2}-&\left(\frac{\alpha-2}2\right)^{\!2},\quad\overline\gamma_\alpha:=\left(\frac{N-2}2\right)^{\!2}+\left(\frac{\alpha-2}2\right)^{\!2},\\
&\qquad\widehat\gamma_\alpha:=\frac{N+\alpha-4}2\,.
\end{split}
\end{equation}
\end{itemize}

\section{Hardy-Rellich inequalities on cones \\ via Emden-Fowler transform}\label{S2}
  
In this Section we provide the optimal constant for the power weighted Hardy-Rellich inequalities \eqref{HardyRellich} for test functions in $C^2_0(C_\Sigma)$  or in suitable subsets of $C^2_0(C_\Sigma)$.  More precisely, we denote by $\Lambda$ a nonempty subset of $\Lambda_\Sigma$ and we consider the following sets
\begin{equation}\begin{split}
\label{ortogonale}
C^2_0(C_\Sigma;\Lambda):=\Big \{ u  \in C^2_0(C_\Sigma)~:~&\int_{\Sigma}u(r\omega)\varphi(\omega)\,d\sigma(\omega)=0\quad\forall r>0\,,\\
&\quad\forall\varphi\in H_\lambda\,,~\forall\lambda\in\Lambda_\Sigma\setminus\Lambda \Big\}
\end{split}\end{equation}
where $H_\lambda$ denotes the eigenspace of $-\Delta_\sigma$ in $H^1_0(\Sigma)$ corresponding to the eigenvalue $\lambda$. If $\Lambda=\Lambda_\Sigma$, then $C^2_0(C_\Sigma;\Lambda_\Sigma)=C^2_0(C_\Sigma)$. Finally, for every $\alpha\in\R$ we define 
\[
\mu_\alpha(C_\Sigma;\Lambda):=\inf_{\scriptstyle u\in C^2_0(C_\Sigma;\Lambda)\atop\scriptstyle u\ne 0}\frac{\displaystyle\int_{C_\Sigma}|x|^\alpha|\Delta u|^2\,dx}{\displaystyle\int_{C_\Sigma}|x|^{\alpha-2}|\nabla u|^2\,dx} \quad \text{and} \quad \mu_\alpha(C_\Sigma):=\mu_\alpha(C_\Sigma;\Lambda_\Sigma)\,.
\]

\begin{Theorem}\label{HR-cone}
Let $\Sigma$  be a $C^2$ domain in $S^{N-1}$ ($N\ge 2$) and let $\Lambda$ be a nonempty subset of $\Lambda_\Sigma$. Then, for every $\alpha\in\R$ 
\begin{equation}\label{best-cone}
\mu_\alpha(C_\Sigma;\Lambda)=\min_{\lambda\in \Lambda}\frac{\left(\gamma_\alpha+\lambda\right)^2}{\widehat\gamma_\alpha^2+\lambda}
\end{equation}
where $\gamma_\alpha$ and $\widehat\gamma_\alpha$ are the values defined in \eqref{ga}, with the agreement that if $0\in\Lambda$ and $\widehat\gamma_\alpha=0$ then $\gamma_\alpha^2/\widehat\gamma_\alpha^2=(N-\alpha)^2/4$. 
\end{Theorem}

\begin{Remark}\label{RN} 
\begin{itemize}[leftmargin=18pt]
\item[(i)]  From \eqref{best-cone}, the Hardy-Rellich inequality 
\[
\int_{C_\Sigma}|x|^\alpha|\Delta u|^2\,dx\ge \mu_\alpha(C_\Sigma;\Lambda) \int_{C_\Sigma}|x|^{\alpha-2}|\nabla u|^2\,dx\quad\forall u\in C^2_0(C_\Sigma;\Lambda)
\]
holds true with a positive constant if and only if $-\gamma_\alpha \not \in \Lambda$. 
\item[(ii)]
When $\Sigma=S^{N-1}$ and $\Lambda=\Lambda_\Sigma$, \eqref{best-cone} reproves \eqref{best-constant} and reads as follows:
\begin{equation}\label{HR-RN}\begin{split}
&\inf_{\scriptstyle u\in C^2_c(\R^N\setminus\{0\})\atop\scriptstyle u\ne 0}\frac{\displaystyle\int_{\R^N}|x|^\alpha|\Delta u|^2\,dx}{\displaystyle\int_{\R^N}|x|^{\alpha-2}|\nabla u|^2\,dx}\\&\qquad=
\min\left\{\left(\frac{N-\alpha}2 \right)^2\right\}\cup\left\{\frac{\left(\gamma_\alpha+j(j+N-2)\right)^2}{\widehat\gamma_\alpha^2+j(j+N-2)}~:~j=1,2,3,...\right\}\,.
\end{split}
\end{equation}
If $\alpha>2-N$, then the weight $|x|^{\alpha-2}$ is locally integrable and, by a density argument, \eqref{HR-RN} holds for test functions in $C^2_c(\R^N)$ as well.
\item[(iii)] 
When $\Sigma=S^{N-1}$ and $\Lambda=\{0\}$, the class $C_0^2(C_\Sigma;\Lambda)$ is the space of radially symmetric functions in $C^2_c(\R^N\setminus \{0\})$ and 
Theorem \ref{HR-cone} states that
\[
\inf_{\scriptstyle u\in C^2_c(\R^N \setminus \{0\})\atop\scriptstyle u\ne 0,~u
\text{ radial}}\frac{\displaystyle\int_{\R^N}|x|^\alpha|\Delta u|^2\,dx}{\displaystyle\int_{\R^N}|x|^{\alpha-2}|\nabla u|^2\,dx}=\left(\frac{N-\alpha}2\right)^2.
\]
We refer to \cite{Cazacu2} for another version of the Hardy-Rellich inequality in $\R^N\setminus\{0\}$ with the radial component of the gradient instead of the full gradient. 
\item[(iv)] 
When $\Sigma=S^{N-1}$ and $\Lambda=\Lambda_\Sigma\setminus\{0\}$, the class $C_0^2(C_\Sigma;\Lambda)$ is the space of functions in $C^2_c(\R^N \setminus \{0\})$ orthogonal to the subspace of radial functions and Theorem \ref{HR-cone} states that
\[
\inf_{\scriptstyle u\in C^2_c(\R^N \setminus \{0\})\setminus\{0\}\atop\scriptstyle \int_{|x|=r}u\,{\dsigma}=0~\forall r>0}\frac{\displaystyle\int_{\R^N}|x|^\alpha|\Delta u|^2\,dx}{\displaystyle\int_{\R^N}|x|^{\alpha-2}|\nabla u|^2\,dx}=\min_{j=1,2,...}\frac{\left(\gamma_\alpha+j(j+N-2)\right)^2}{\widehat\gamma_\alpha^2+j(j+N-2)}\,.
\]
\item[(vi)] 
The formulas stated in items (iii) and (iv) suggest that, in general, no symmetrization or rearrangement arguments can be applied to study the minimization problem defining $\mu_\alpha(C_\Sigma;\Lambda)$, as the infimum when restricting to radial functions may be larger. Indeed, for example, for $\alpha=0$ and $N=4$, the infimum in the formula stated in (iii) equals $4$, while the one in (iv) equals $3$.The failure of the Szeg\H{o} method is due to the presence of second order derivatives. 
\end{itemize}

\end{Remark}
\noindent
The proof of Theorem \ref{HR-cone} relies on two key tools: 
\begin{itemize}[leftmargin=15pt]
\item[1)]
the Emden-Fowler (or logarithmic-cylindri\-cal) change of variables, which transforms power weighted $L^2$-norms of a mapping and of its derivatives into corresponding weight-free $L^2$-norms with constant coefficients. This is particularly useful for dilation-invariant domains, such as (positive) cones. 
\item[2)] Fourier decomposition with respect to a Hilbertian basis of $H^1_0(\Sigma)$ made by eigenfunctions of the Laplace-Beltrami operator. 
\end{itemize}
The Emden-Fowler change of variables works as follows: given a $C^2$ domain $\Sigma$ in $S^{N-1}$, consider the corresponding cylinder $Z_\Sigma$ defined in \eqref{ZSigma} and the space $C^2_0(Z_\Sigma)$ defined in \eqref{C2Z}.
For every $u\colon\overline{C_\Sigma}\setminus\{0\}\to\R$ let $w\colon\overline{Z_\Sigma}\to\R$ be defined by
\begin{equation}
\label{Emden-Fowler}
u(x)=|x|^{\frac{4-N-\alpha}{2}}w\left(-\log|x|,\frac{x}{|x|}\right).
\end{equation}
The correspondence $u\mapsto w=T(u)$ defines the \emph{Emden-Fowler transform}.
\begin{Lemma}
One has that $u\in C^2_0(C_\Sigma)$ if and only if $w\in C^2_0(Z_\Sigma)$. 
Moreover:
\begin{gather}
\label{w2}
\int_{C_\Sigma}|x|^{\alpha-4}|u|^2\,dx=\int_{Z_\Sigma}|w|^2\,dt\,{\dsigma}\\
\label{first-derivative-w}
\int_{C_\Sigma}|x|^{\alpha-2}|\nabla u|^2\,dx=\int_{Z_\Sigma}\left(Q_\alpha(w)+|w_t|^2\right)\,dt\,{\dsigma}\\
\label{second-derivative-w}
\int_{C_\Sigma}|x|^{\alpha}|\Delta u|^2\,dx=\int_{Z_\Sigma}\left(|L_\alpha w|^2+|w_{tt}|^2+2|\nabla_\sigma w_t|^2+2\overline\gamma_\alpha|w_t|^2\right)\,dt\,{\dsigma}
\end{gather}
where
\begin{equation}\label{QL}
Q_\alpha(w)=|\nabla_\sigma w|^2+\widehat\gamma_\alpha^2|w|^2,\qquad L_\alpha w=-\Delta_\sigma w+\gamma_\alpha w\,,
\end{equation}
$\gamma_\alpha$, $\overline\gamma_\alpha$, and $\widehat\gamma_\alpha$ are defined in \eqref{ga}.
\end{Lemma}

\begin{proof}
The first part of the statement as well as \eqref{w2} and \eqref{second-derivative-w} have been proved in \cite[Lemma 3.4]{CM}. The proof of \eqref{first-derivative-w} can be carried out in a similar way (see also \cite{CM2}). 
\end{proof}

\noindent
\begin{proof}[Proof of Theorem \ref{HR-cone}]
Let $u\in C^2_0(C_\Sigma;\Lambda)$ and let $w=T(u)\in C^2_0(Z_\Sigma)$ be the Emden-Fowler transform of $u$. Then, by \eqref{ortogonale}, 
\begin{equation}\label{ortogonaleZ}
\int_{\Sigma}w(t,\omega)\varphi(\omega)\,d\sigma(\omega)=0\quad  \forall t\in\R\,,  ~\forall\varphi\in H_\lambda\,,~\forall\lambda\in\Lambda_\Sigma\setminus\Lambda\,.
\end{equation}
Let $F(w)$ and $G(w)$ be the right hand sides of \eqref{second-derivative-w} and \eqref{first-derivative-w}, respectively.
Let $\{\varphi_j\}_{j\in\mathbb{N}}$ be a Hilbertian basis of $H^1_0(\Sigma)\cap H^2(\Sigma)$ made by eigenfunctions of $-\Delta_\sigma$ and define:
\begin{equation}\label{proiez}
y_j(t)=\int_\Sigma \varphi_j(\omega)w(t,\omega)\,d\sigma(\omega)\,,\quad w_j(t,\omega)=y_j(t)\varphi_j(\omega)\quad(j\in\mathbb{N}).
\end{equation}
If $\varphi_i$ and $\varphi_j$ are distint eigenfunctions, then
\[
\int_\Sigma\varphi_i\varphi_j\,{\dsigma}=\int_\Sigma\nabla_\sigma\varphi_i\cdot\nabla_\sigma\varphi_j\,{\dsigma}=0
\quad\text{and}\quad \int_\Sigma|\nabla_\sigma\varphi_j|^2\,{\dsigma}=\lambda_j\int_\Sigma|\varphi_j|^2\,{\dsigma}.
\]
Recalling \eqref{ortogonaleZ}, this directly implies that  
\begin{equation}\label{FG-decomposition}
F(w)=\sum_{j\in I_\Lambda}F(w_j)\quad\text{and}\quad G(w)=\sum_{j\in I_\Lambda}G(w_j)
\end{equation}
where $I_\Lambda=\{j\in\mathbb{N}~:~\lambda_j\in  \Lambda  \}$. 
We now show that
\begin{equation}\label{j-estimate}
{F(w_j)}\ge\frac{\left(\lambda_j+\gamma_\alpha\right)^2}{\lambda_j+\widehat\gamma_\alpha^2}{G(w_j)}\,.
\end{equation}
To prove \eqref{j-estimate}, for every $y\in C^2_c(\R)$, denoting $y_\varepsilon(t)=y(\varepsilon t)$, we observe that
\[
\inf_{\scriptstyle y\in C^2_c(\R)\atop\scriptstyle y\ne 0}\frac{F(y\varphi_j)}{G(y\varphi_j)}=
\inf_{\scriptstyle y\in C^2_c(\R)\setminus\{0\}\atop\scriptstyle \varepsilon>0}\frac{F(y_\varepsilon\varphi_j)}{G(y_\varepsilon\varphi_j)}.
\]
Now, we compute
\begin{equation}\label{f-epsilon}
\frac{F(y_\varepsilon\varphi_j)}{G(y_\varepsilon\varphi_j)}=\frac{\displaystyle(\lambda_j+\gamma_\alpha)^2\int_\R|y|^2\,dt+2\varepsilon^2(\lambda_j+\overline\gamma_\alpha)\int_\R|y'|^2\,dt+\varepsilon^4\int_\R|y''|^2\,dt}{\displaystyle(\lambda_j+\widehat\gamma_\alpha^2)\int_\R|y|^2\,dt+\varepsilon^2\int_\R|y'|^2\,dt}.
\end{equation}
Let us study the right hand side in \eqref{f-epsilon} as a function of $\varepsilon^2$. In general, for a mapping of the form
\[
f(\varepsilon)=\frac{A+2\varepsilon B+C\varepsilon^2}{D+\varepsilon E}
\]
with $A,...,E>0$, elementary calculus shows that $f'(\varepsilon)>0$ for every $\varepsilon>0$ if $2BD-AE\ge 0$. For the specific case $f(\varepsilon^2)$, defined by the right-hand side of \eqref{f-epsilon}, we calculate:

\begin{equation}
\begin{split}
2BD-AE&=\left[2(\lambda_j+\overline\gamma_\alpha)(\lambda_j+\widehat\gamma_\alpha^2)-(\lambda_j+\gamma_\alpha)^2\right]\int_\R|y|^2\,dt\int_\R|y'|^2\,dt\\
\label{2BD-AE}
&=p_\alpha(\lambda_j)\int_\R|y|^2\,dt\int_\R|y'|^2\,dt
\end{split}
\end{equation}
where
\begin{equation}\label{p-alpha}
p_\alpha(\lambda)=\lambda^2+\left[\left(\alpha-2\right)^2+2\,\widehat\gamma_\alpha^2\right]\lambda+\widehat\gamma_\alpha^4
\end{equation}
because
\begin{equation*}
\overline\gamma_\alpha+\widehat\gamma_\alpha^2-\gamma_\alpha=2\left(\frac{\alpha-2}2\right)^2+\widehat\gamma_\alpha^2\ge 0\quad\text{and}\quad 2\overline\gamma_\alpha\widehat\gamma_\alpha^2-\gamma_\alpha^2=\widehat\gamma_\alpha^4\ge 0\,.
\end{equation*}
Since $\lambda_j\ge 0$, by \eqref{2BD-AE} and \eqref{p-alpha}, we obtain that $2BD-AE\ge 0$. Thus, $f$ is increasing on $[0,\infty)$ and then
\[
f(\varepsilon^2)> f(0)=\frac{(\lambda_j+\gamma_\alpha)^2}{\lambda_j+\widehat\gamma_\alpha^2}\quad\forall\varepsilon>0\,.
\]
This establishes \eqref{j-estimate}.  Using \eqref{first-derivative-w}--\eqref{second-derivative-w}, \eqref{FG-decomposition} and \eqref{j-estimate}, we conclude:
\begin{equation}\label{crucial_ineq}
\begin{split}
\int_{C_\Sigma}|x|^\alpha|\Delta u|^2\,dx=F(w)&=\sum_{j\in I_\Lambda}F(w_j)\ge\sum_{j\in I_\Lambda}\frac{\left(\lambda_j+\gamma_\alpha\right)^2}{\lambda_j+\widehat\gamma_\alpha^2}\,G(w_j)\\
&\ge\min_{\lambda\in\Lambda}\frac{\left(\lambda+\gamma_\alpha\right)^2}{\lambda+\widehat\gamma_\alpha^2}\sum_{j\in\mathbb{N}}G(w_j)=\min_{\lambda\in\Lambda}\frac{\left(\lambda+\gamma_\alpha\right)^2}{\lambda+\widehat\gamma_\alpha^2}\,G(w)\\
&=\min_{\lambda\in\Lambda}\frac{\left(\lambda+\gamma_\alpha\right)^2}{\lambda+\widehat\gamma_\alpha^2}\int_{C_\Sigma}|x|^{\alpha-2}|\nabla u|^2\,dx\,.
\end{split}
\end{equation}
Hence, $\mu_\alpha(C_\Sigma;\Lambda)\ge\min_{\lambda\in\Lambda}\frac{\left(\lambda+\gamma_\alpha\right)^2}{\lambda+\widehat\gamma_\alpha^2}$. Let us establish the opposite inequality. Consider $\overline\lambda\in\Lambda$ such that 
\begin{equation}\label{lambdasegnato}
\frac{\left(\overline\lambda+\gamma_\alpha\right)^2}{\overline\lambda+\widehat\gamma_\alpha^2}=\min_{\lambda\in\Lambda}\frac{\left(\lambda+\gamma_\alpha\right)^2}{\lambda+\widehat\gamma_\alpha^2}.
\end{equation} 
Let $\varphi\in H_{\overline\lambda}$, and fix $y\in C^2_c(\R)$. For every $\varepsilon>0$, define $w_\varepsilon\colon\overline{Z_\Sigma}\to\R$ by
\[
w_\varepsilon(t,\omega)=y(\varepsilon t)\varphi(\omega).
\]
Then $w_\varepsilon\in C^2_0(Z_\Sigma)$ and $u_\varepsilon:=T^{-1}(w_\varepsilon)\in C^2_0(C_\Sigma;\Lambda)$. Using the same computations as before, we find:
\begin{equation}
\label{sviluppi-epsilon}
\begin{split}
&\frac{\displaystyle\int_{C_\Sigma}|x|^\alpha|\Delta u_\varepsilon|^2\,dx}{\displaystyle\int_{C_\Sigma}|x|^{\alpha-2}|\nabla u_\varepsilon|^2\,dx}=
\frac{F(y_\varepsilon\varphi)}{G(y_\varepsilon\varphi)}\\
&\quad=\frac{\displaystyle(\overline\lambda+\gamma_\alpha)^2\int_\R|y|^2\,dt+2\varepsilon^2(\overline\lambda+\overline\gamma_\alpha)\int_\R|y'|^2\,dt+\varepsilon^4\int_\R|y''|^2\,dt}{\displaystyle(\overline\lambda+\widehat\gamma_\alpha^2)\int_\R|y|^2\,dt+\varepsilon^2\int_\R|y'|^2\,dt}.
\end{split}
\end{equation}
Taking the limit as $\varepsilon\to 0$, and invoking \eqref{lambdasegnato}, we obtain that 
$\mu_\alpha(C_\Sigma;\Lambda)\le\min_{\lambda\in\Lambda}\frac{\left(\lambda+\gamma_\alpha\right)^2}{\lambda+\widehat\gamma_\alpha^2}$. 
\end{proof}

\begin{Remark}
\label{R:principal}
If
\begin{equation}
\label{principal}
\frac{8-N}{3}\leq \alpha \leq N
\end{equation} 
the mapping $\lambda\mapsto \frac{\left(\gamma_\alpha+\lambda \right)^2}{\widehat\gamma_\alpha^2+\lambda}$ turns out to be increasing in $(0,\infty)$.
Then,  for all $\Lambda$ nonempty subset of $\Lambda_\Sigma$, if we set $\lambda_{\Lambda}:=\min\Lambda \ge 0$, we have
\begin{equation}
\label{mu-alpha-principal}
\mu_\alpha(C_\Sigma;\Lambda)=
\frac{(\lambda_{\Lambda}+\gamma_\alpha)^2}{\lambda_{\Lambda}+\widehat\gamma_\alpha^2}\,.
\end{equation}
Clearly, if $\Lambda= \Lambda_{\Sigma}$, then $\lambda_{\Lambda}=\lambda_\Sigma$. In particular, if \eqref{principal} holds true and we denote by $\Lambda_{k}=\{j(j+N-2)\}_{j\geq k}$ for some positive integer $k$, recalling \eqref{HR-RN}, then
\begin{equation}
\label{mualpha-Rn}
\frac{\left(\gamma_\alpha+k(k+N-2)\right)^2}{\widehat\gamma_\alpha^2+k(k+N-2)}=\mu_\alpha(\R^N\setminus\{0\},\Lambda_{k})>\mu_\alpha(\R^N\setminus\{0\})=\left(\frac{N-\alpha}2\right)^2\,.
\end{equation}

We point out that \eqref{principal} is not necessary for \eqref{mu-alpha-principal} to hold. For example, in the weight-free case, i.e., $\alpha=0$, and when $\Lambda= \Lambda_{\Sigma}$, \eqref{mualpha-Rn} holds if and only if $N\ge 5$. Moreover, for $\alpha=0$ and $\Sigma$ the half-sphere, a direct inspection reveals that \eqref{mu-alpha-principal} with $\Lambda=\Lambda_\Sigma$ holds true in any dimension $N\ge 2$ (see \cite[Proposition 4.5]{CM}). Therefore, denoting by $\R^N_+$ any half-space, since  $\lambda_\Sigma=N-1$, we infer that
\begin{equation*}
\mu_0(\R^N_+)= \frac{(N^2-4)^2}{4[(N-2)^2+8]}  \quad\forall N\ge 2\,.
\end{equation*}
\end{Remark}

\begin{Remark}
From \cite{CM}, we recall that 
\begin{equation}\label{Rellich-cones}
\inf_{\scriptstyle u\in C^2_0(C_\Sigma)\atop\scriptstyle u\ne 0}\frac{\displaystyle\int_{C_\Sigma}|x|^\alpha|\Delta u|^2\,dx}{\displaystyle\int_{C_\Sigma}|x|^{\alpha-4}|u|^2\,dx}=\min_{\lambda\in \Lambda_\Sigma}{\left(\gamma_\alpha+\lambda\right)^2}\,.
\end{equation}
Hence, for proper cones $C_\Sigma\ne\R^N\setminus\{0\}$,  the Hardy-Rellich inequality in $C^2_0(C_\Sigma)$ holds true with a positive constant (i.e., $\mu_\alpha(C_\Sigma)>0$) if and only if the corresponding Rellich inequality on cones does. In the case  $C_\Sigma=\R^N\setminus\{0\}$ (i.e., $\Sigma=S^{N-1}$), if  the Rellich inequality in $C^2_c(\R^N\setminus\{0\})$ holds true with a positive constant, then also the Hardy-Rellich inequality does. However, the converse is not true: when $\gamma_\alpha=0$ (i.e., $\alpha= 4-N$), the infimum in \eqref{Rellich-cones} is 0, whereas the optimal Hardy-Rellich constant is $\min\{(N-2)^2,N-1\}$ which is positive for all $N>2$.
\end{Remark}

 As soon as $\mu_\alpha(C_\Sigma)>0$, one can define the Sobolev space $H^2\cap H^1_0(C_\Sigma;|x|^\alpha)$ as the completion of $C^2_0(C_\Sigma)$ with respect to the norm
\begin{equation*}
\|u\|_{2,\alpha}:=\left(\int_{C_\Sigma}|x|^\alpha|\Delta u|^2\,dx\right)^{\frac12}\,.
\end{equation*}
By density, the Hardy-Rellich inequality holds true for every $u\in H^2\cap H^1_0(C_\Sigma;|x|^\alpha)$. 
We now show that the optimal constant \eqref{best-cone} is never attained in $H^2\cap H^1_0(C_\Sigma;|x|^\alpha)$. 
This outcome is not surprising, since we deal with a minimization problem involving a ratio of integrals with the same homogeneity, which is invariant with respect to the action of a noncompact group (dilations). This issue was previously addressed in \cite{TZ} and \cite{Cazacu} for the entire space when $\alpha=0$. Here we discuss the more general weighted case on cones, employing a completely different technique.

\begin{Theorem}\label{not_att}
Under the same assumptions of Theorem \ref{HR-cone}, if $\mu_\alpha(C_\Sigma)>0$, then $\mu_\alpha(C_\Sigma;\Lambda)$ is never attained if considered in the space of mappings $u\in H^2\cap H^1_0(C_\Sigma;|x|^\alpha)$ satisfying
\begin{equation}
\label{ortogonale-quasi-ovunque}
\int_{\Sigma}u(r\omega)\varphi(\omega)\,d\sigma(\omega)=0\quad\text{for a.e. } r>0\,,~\forall\varphi\in H_\lambda\,,~\forall\lambda\in\Lambda_\Sigma\setminus\Lambda
\end{equation}
where $H_\lambda$ is the eigenspace of $-\Delta_\sigma$ in $H^1_0(\Sigma)$ corresponding to the eigenvalue $\lambda$. 
\end{Theorem}

\begin{proof} 
The Emden-Fowler transform $u\mapsto T(u)=w$ given by \eqref{Emden-Fowler} defines an isomorphism between the spaces $H^2\cap H^1_0(C_\Sigma;|x|^\alpha)$ and $H^2\cap H^1_0(Z_\Sigma)$, see \cite[Lemma 3.5]{CM}. Moreover \eqref{w2}--\eqref{second-derivative-w} hold true also for every $u\in H^2\cap H^1_0(C_\Sigma;|x|^\alpha)$. Assume,  by contradiction, that the infimum $\mu_\alpha(C_\Sigma;\Lambda)$ is attained by some $\overline{u}\in H^2\cap H^1_0(C_\Sigma;|x|^\alpha)$, $\overline{u}\ne 0$, satisfying \eqref{ortogonale-quasi-ovunque}. Let $\overline{w}=T(\overline{u})\in H^2\cap H^1_0(Z_\Sigma)$. 
Expanding $\overline{w}$ as in \eqref{proiez}, writing \eqref{crucial_ineq} for $w=\overline{w}$, 
we infer that all inequalities in  \eqref{crucial_ineq} must be equalities. Hence, we get
\[
\sum_{j\in I_\Lambda}\frac{\left(\lambda_j+\gamma_\alpha\right)^2}{\lambda_j+\widehat\gamma_\alpha^2}\,G(\overline w_j)=\min_{\lambda\in\Lambda}\frac{\left(\lambda+\gamma_\alpha\right)^2}{\lambda+\widehat\gamma_\alpha^2}\sum_{j\in _\Lambda}G(\overline w_j)=\frac{\left(\overline\lambda+\gamma_\alpha\right)^2}{\overline\lambda+\widehat\gamma_\alpha^2}\sum_{j\in _\Lambda}G(\overline w_j)\,,
\]
where $\overline\lambda \in \Lambda$ is the eigenvalue attaining the above minimum. In particular, the above identity makes sense only if  
\[
\overline{w}(t,\omega)=\sum_{j \in J_{\overline\lambda}}\overline w_{j}(t,\omega)=\sum_{j \in J_{\overline\lambda}} \overline y_{j}(t)\varphi_{j}(\omega)\,,
\]
where $ J_{\overline\lambda}=\{j \in \mathbb{N}:\varphi_{j} \text{ is an eigenfunction with eigevalue }\overline\lambda \}$. Clearly, the set $ J_{\overline\lambda}$ is finite and nonempty.
 For $0<\varepsilon<1$, let now 
 \[
 w^{\varepsilon}(t,\omega)=\sum_{j \in J_{\overline\lambda}} w^{\varepsilon}_j(t,\omega)=\sum_{j \in J_{\overline\lambda}} \overline y_{j}(\varepsilon t)\varphi_{j}(\omega)\,.
 \]  
 By slightly modifying the computations in \eqref{sviluppi-epsilon} and exploited to estimate \eqref{f-epsilon}, for $0<\varepsilon<1$ we get
 \begin{equation}\label{contrad}
\begin{split} 
& \mu_\alpha(C_\Sigma;\Lambda) \leq \frac{F(w^{\varepsilon})}{G(w^{\varepsilon})} =\frac{\sum_{j \in J_{\overline\lambda}}F(w_j^{\varepsilon})}{\sum_{j \in J_{\overline\lambda}}G(w_j^{\varepsilon})}=
 \\
&\frac{\displaystyle(\overline\lambda+\gamma_\alpha)^2\sum_{j \in J_{\overline\lambda}}\int_\R| \overline y_j|^2\,dt+2\varepsilon^2(\overline\lambda+\overline\gamma_\alpha)\sum_{j \in J_{\overline\lambda}}\int_\R | \overline y_j'|^2\,dt+\varepsilon^4 \sum_{j \in J_{\overline\lambda}} \int_\R | \overline y_j''|^2\,dt}{\displaystyle(\overline\lambda+\widehat\gamma_\alpha^2)\sum_{j \in J_{\overline\lambda}} \int_\R| \overline y_j|^2\,dt+\varepsilon^2 \sum_{j \in J_{\overline\lambda}}\int_\R | \overline y_j'|^2\,dt} \\
  &\displaystyle<\frac{\sum_{j \in J_{\overline\lambda}}F(\overline w_j)}{\sum_{j \in J_{\overline\lambda}}G(\overline w_j)}=\frac{F(\overline{w})}{G(\overline{w})} = \mu_\alpha(C_\Sigma;\Lambda)
\end{split}
\end{equation}
 where the strict inequality comes by exploiting the same strictly increasing function $f(\varepsilon)=\frac{A+2\varepsilon B+C\varepsilon^2}{D+\varepsilon E}$ introduced in the proof of Theorem \ref{HR-cone} with a slightly different choice of the coefficients $A,...,E>0$ (since $\overline{w} \neq 0$) still satisfying the condition $2BD-AE\ge 0$ which indeed here reads
\[
2BD-AE=p_\alpha(\overline\lambda)\left( \sum_{j \in J_{\overline\lambda}}\int_\R| \overline y_j|^2\,dt\right)\left( \sum_{j \in J_{\overline\lambda}}\int_\R| \overline y'_j|^2\,dt\right)
\]
with $p_\alpha$ as in \eqref{p-alpha}. In particular $p_\alpha(\overline\lambda)\ge 0$. 
Finally, \eqref{contrad} gives a contradiction and concludes the proof. 
\end{proof}

\begin{Remark}[Hardy-Rellich inequality versus Rellich inequality] 
Since $\gamma_\alpha=\gamma_{4-\alpha}$, the optimal Rellich constant \eqref{Rellich-cones} is the same with respect to the reflection $\alpha\mapsto 4-\alpha$ about 2. This is not true for the optimal Hardy-Rellich constant \eqref{best-cone}, since in general $\widehat\gamma^2_\alpha\ne\widehat\gamma^2_{4-\alpha}$. In fact, the deep reason is related to the Kelvin transform. More precisely, for every $u\colon\overline{C_\Sigma}\setminus\{0\}\to\R$, let $\hat{u}\colon\overline{C_\Sigma}\setminus\{0\}\to\R$ be defined by
\[
\hat{u}(x):=|x|^{2-N}u\left(\frac x{|x|^2}\right).
\]
Then, noticing that the orthogonality condition \eqref{ortogonale} does not involve the radial variable, $u\in C^2_0(C_\Sigma; \Lambda)$ if and only if $\hat{u}\in C^2_0(C_\Sigma; \Lambda)$. Moreover:
\begin{equation}\label{-alpha1}
\int_{C_\Sigma}|x|^{\alpha-4}|\hat{u}|^2\,dx=\int_{C_\Sigma}|x|^{-\alpha}|u|^2\,dx\,,\qquad\int_{C_\Sigma}|x|^{\alpha}|\Delta \hat{u}|^2\,dx=\int_{C_\Sigma}|x|^{4-\alpha}|\Delta u|^2\,dx
\end{equation}
whereas
\begin{equation}\label{-alpha2}
\begin{split}
&\int_{C_\Sigma}|x|^{\alpha-2}|\nabla\hat{u}|^2 \,dx=\int_{C_\Sigma}|x|^{2-\alpha}|\nabla u|^2\,dx +(N-2)(\alpha-2)\int_{C_\Sigma} |x|^{-\alpha}|u|^2\,dx\,. 
\end{split}
\end{equation}
\end{Remark}
As a consequence of \eqref{-alpha1}--\eqref{-alpha2} and Theorem \ref{HR-cone}, we plainly obtain the following Schmincke's type inequality \cite{SC} on cones:

\begin{Theorem}\label{T:Schmincke-type}
Let $\Sigma$ be a $C^2$ domain in $S^{N-1}$ with $N\ge 3$ and let $\Lambda$ be a nonempty subset of $\Lambda_{\Sigma}$. Then, for every $\alpha\in\R$ and for every $u\in C^2_0(C_\Sigma; \Lambda)$
\begin{equation}\label{Sch}
\begin{split}
&\int_{C_\Sigma}|x|^\alpha|\Delta u|^2\,dx\ge\\
&\mu_{4-\alpha}(C_\Sigma;\Lambda)  \left[\int_{C_\Sigma}|x|^{\alpha-2}|\nabla u|^2\,dx+(N-2)(2-\alpha)\int_{C_\Sigma}|x|^{\alpha-4}|u|^2\right]\,dx.
\end{split}
\end{equation}
Furthermore, for $\alpha\in \R$ fixed, the constant $\mu_{4-\alpha}(C_\Sigma;\Lambda)$ cannot be replaced by a larger one. 
\end{Theorem}
We point out that if $\alpha\le 2$ then $\mu_{4-\alpha}(C_\Sigma; \Lambda)\le\mu_\alpha(C_\Sigma; \Lambda)$, that is, a positive extra term involving the weighted $L^2$-norm can be added on the right hand side but the constant in front of the term with the gradient does not reach $\mu_\alpha(C_\Sigma; \Lambda)$.

\begin{Remark}[] 
Let us consider the case $\Sigma = S^{N-1}$, where $C_\Sigma = \mathbb{R}^N \setminus \{0\}$. Furthermore, assume that $4-N \leq \alpha \leq \frac{N+4}{3}$. Then $\mu_{4-\alpha}(\mathbb{R}^N \setminus \{0\}) = \left(\frac{N+4-\alpha}{2}\right)^2$ (see Remark \ref{R:principal}). By using \eqref{Sch} with $\Lambda = \Lambda_{\Sigma}$, it follows that for every $u\in C^2_c(\R^N\setminus\{0\})$
\begin{equation*}
\begin{split}
&\int_{\R^{N}}|x|^\alpha|\Delta u|^2\,dx\ge\\
& \left(\frac{N+\alpha-4}{2} \right)^2 \left[\int_{\R^{N}}|x|^{\alpha-2}|\nabla u|^2\,dx+(N-2)(2-\alpha)\int_{\R^{N}}|x|^{\alpha-4}|u|^2\,dx\right].
\end{split}
\end{equation*}
This is a particular case of the Schmincke's inequality proved in \cite[Corollary 6]{Be} and in \cite[Theorems 2.7]{GPPS}. Theorem \ref{T:Schmincke-type} shows its optimality in the specified sense. 
\end{Remark}

 \subsection{Some related optimal inequalities} 
 
 According to Theorem \ref{HR-cone}, the optimal constant $\mu_\alpha(C_\Sigma;\Lambda)$ in the Hardy-Rellich inequality on cones depends only on the parameter $\alpha$, the domain $\Sigma$ in $S^{N-1}$, and the subset $\Lambda\subseteq \Lambda_\Sigma$. This observation is further highlighted by Theorem \ref{3inf} below. Recalling \eqref{QL}, we define:

\begin{gather*} 
\mu_\alpha(Z_\Sigma;\Lambda):=\inf_{\scriptstyle w\in C^2_0(Z_\Sigma);\Lambda)\atop \scriptstyle w\ne 0}\frac{\displaystyle\int_{Z_\Sigma}|L_\alpha w|^2\,dt\,d\sigma}{\displaystyle\int_{Z_\Sigma}Q_\alpha(w)\,dt\,d\sigma}\,;\,\,
\mu_\alpha(\Sigma;\Lambda):=\inf_{\scriptstyle \varphi\in C^2_0(\Sigma;\Lambda)\atop \scriptstyle \varphi\ne 0}\frac{\displaystyle\int_{\Sigma}|L_\alpha \varphi|^2\,d\sigma}{\displaystyle\int_{\Sigma}Q_\alpha(\varphi)\,d\sigma}\,.
\end{gather*}
where $C^2_0(Z_\Sigma);\Lambda)$ and $ C^2_0(\Sigma;\Lambda)$ are defined as the subsets of $C^2_0(Z_\Sigma)$ and $ C^2_0(\Sigma)$ (see \eqref{C2Z}), respectively, of functions satisfying the orthogonality condition \eqref{ortogonaleZ} (with no dependence on $t$ in the second case).
The relationship between the above constants and $\mu_\alpha(C_\Sigma;\Lambda)$ is given by the following statement.
\begin{Theorem}\label{3inf}
Let $\Sigma$ be a $C^2$ domain in $S^{N-1}$ with $N\ge 2$ and let $\Lambda$ be a nonempty subset of $\Lambda_{\Sigma}$. For every $\alpha \in \R$ there holds
\begin{equation} \label{crucial=}
\mu_\alpha(Z_\Sigma;\Lambda)=\mu_\alpha(\Sigma;\Lambda)=\min_{\lambda\in \Lambda}\frac{\left(\gamma_\alpha+\lambda\right)^2}{\widehat\gamma_\alpha^2+\lambda}.
\end{equation}
Moreover the infimum in $\mu_\alpha(\Sigma;\Lambda)$, if considered in the space of test functions in $H^1_0(\Sigma)\cap H^2(\Sigma)$ satisfying the orthogonality condition \eqref{ortogonaleZ}, is attained by eigenfunctions corresponding to the eigenvalue minimizing the right hand side in \eqref{crucial=}.
\end{Theorem}
By combining Theorem \ref{3inf} with Theorem \ref{HR-cone}, we infer that
\begin{equation}
\label{mu-alpha-Sigma}
\mu_\alpha(C_\Sigma;\Lambda)=\mu_\alpha(\Sigma;\Lambda)\,. 
\end{equation}
Hence, for the remainder of this paper, we will simply refer to $\mu_\alpha(\Sigma;\Lambda)$.
\begin{proof}
 For every $w\in C^2_0(Z_\Sigma;\Lambda)$ and for every $t\in\R$ one has that $w(t,\cdot)\in C^2_0(\Sigma;\Lambda)$ and then
\[
\int_\Sigma|L_\alpha w(t,\omega)|^2\,d\sigma(\omega)\ge \mu_\alpha(\Sigma;\Lambda)\int_\Sigma Q_\alpha( w(t,\omega))^2\,d\sigma(\omega).
\]
Integrating over $t\in\R$ we obtain that $\mu_\alpha(Z_\Sigma;\Lambda)\ge\mu_\alpha(\Sigma;\Lambda)$. Let us prove the opposite inequality. Let $\psi\in C^2_c(\R)\setminus\{0\}$. For every $\varphi\in C^2_0(\Sigma;\Lambda)\setminus\{0\}$ set $w(t,\sigma)=\psi(t)\varphi(\sigma)$. Then $w\in C^2_0(Z_\Sigma;\Lambda)$ and 
\[
\mu_\alpha(Z_\Sigma;\Lambda)\le \frac{\displaystyle\int_{Z_\Sigma}|L_\alpha w|^2\,dt\,d\sigma}{\displaystyle\int_{Z_\Sigma}Q_\alpha(w)\,dt\,d\sigma}=\frac{\displaystyle\int_{\Sigma}|L_\alpha \varphi|^2\,d\sigma}{\displaystyle\int_{\Sigma}Q_\alpha(\varphi)\,d\sigma}\,.
\]
Hence, $\mu_\alpha(Z_\Sigma;\Lambda)\le \mu_\alpha(\Sigma;\Lambda)$. In conclusion, $\mu_\alpha(Z_\Sigma;\Lambda)=\mu_\alpha(\Sigma;\Lambda)$. 

Finally, we show that $\mu_\alpha(\Sigma;\Lambda)=\min_{\lambda\in \Lambda}\frac{\left(\gamma_\alpha+\lambda\right)^2}{\widehat\gamma_\alpha^2+\lambda}$ and the infimum in the definition of $\mu_\alpha(\Sigma;\Lambda)$  is attained by eigenfunctions corresponding to the eigenvalue minimizing the right hand side in \eqref{crucial=}. To this aim, let $\lambda\in\Lambda$ and let $\varphi\in H^1_0(\Sigma)$ be a corresponding eigenvector, then
\[
\int_\Sigma|L_\alpha\varphi|^2\,d\sigma=\int_\Sigma\left|(\lambda+\gamma_\alpha)\varphi\right|^2\,d\sigma
\text{~~and~~}
\int_\Sigma Q_\alpha(\varphi)\,d\sigma=\int_\Sigma\left(\lambda+\widehat\gamma_\alpha^2\right)|\varphi|^2\,d\sigma.
\]
Hence,
\begin{equation}
\label{inequality}
\mu_\alpha(\Sigma;\Lambda)\le\frac{\left(\gamma_\alpha+\lambda\right)^2}{\widehat\gamma_\alpha^2+\lambda}.
\end{equation}
Using the same notation introduced in the proof of Theorem \ref{HR-cone}, for every $\varphi\in H^1_0(\Sigma)\cap H^2(\Sigma)$ satisfying the orthogonality condition \eqref{ortogonaleZ}, we have that
$\varphi=\sum_{j\in I_{\Lambda}}c_j\varphi_j$ where $c_j=\int_\Sigma\varphi_j\varphi\,d\sigma\quad(j\in\mathbb{N})$. Therefore, for every $k\in\mathbb{N}$ one has
\[\begin{split}
\int_{\Sigma}\bigg|L_\alpha \bigg(\sum_{\scriptstyle j\in I_{\Lambda} \atop \scriptstyle j\leq k} c_j\varphi_j\bigg)\bigg|^2\,d\sigma&=\sum_{\scriptstyle j\in I_{\Lambda} \atop \scriptstyle j\leq k}|c_j|^2 \left(\lambda_j+\gamma_\alpha\right)^2\int_{\Sigma}|\varphi_j|^2\,d\sigma\\
& \ge\min_{\lambda\in\Lambda }\frac{\left(\lambda+\gamma_\alpha\right)^2}{\lambda+\widehat\gamma_\alpha^2}\sum_{\scriptstyle j\in I_{\Lambda} \atop \scriptstyle j\leq k}|c_j|^2\left(\lambda_j+\widehat\gamma_\alpha^2\right)\int_{\Sigma}|\varphi_j|^2\,d\sigma
\end{split}
\]
and
\[
\int_{\Sigma}Q_\alpha\bigg(\sum_{\scriptstyle j\in I_{\Lambda} \atop \scriptstyle j\leq k} c_j\varphi_j\bigg)\,d\sigma
=\sum_{\scriptstyle j\in I_{\Lambda} \atop \scriptstyle j\leq k}|c_j|^2\left(\lambda_j+\widehat\gamma_\alpha^2\right)\int_{\Sigma}|\varphi_j|^2\,d\sigma.
\]
Then, passing to the limit as $k\to\infty$ one obtains
\[
\frac{\displaystyle\int_{\Sigma}|L_\alpha\varphi|^2\,d\sigma}{\displaystyle\int_{\Sigma}Q_\alpha(\varphi)\,d\sigma}\ge \min_{\lambda\in\Lambda}\frac{\left(\lambda+\gamma_\alpha\right)^2}{\lambda+\widehat\gamma_\alpha^2}.
\]
This holds true for every $\varphi\in H^1_0(\Sigma)\cap H^2(\Sigma)$ satisfying the orthogonality condition \eqref{ortogonaleZ}, and thus, together with \eqref{inequality}, concludes the proof. 
\end{proof}

\section{The Dirichlet case}\label{S3}

Let $\Sigma$ be a \emph{proper} $C^2$ subdomain of $S^{N-1}$. In this context, it is meaningful to study the Hardy-Rellich inequality within the two distinct spaces defined in \eqref{C20}: either considering functions in the space $C^2_0(C_\Sigma)$ (the \emph{Navier case}) or in the \emph{proper} space $C^2_c(C_\Sigma)$ (the \emph{Dirichlet case}). While the Navier case has already been addressed in the previous section, here we focus on the Dirichlet case.

Let $C^2_0(C_\Sigma;\Lambda)$ be as defined in \eqref{ortogonale}, we set $C^2_c(C_\Sigma;\Lambda):=C^2_0(C_\Sigma;\Lambda) \cap C^2_c(C_\Sigma)$ and we define: 
\begin{equation*}
\mu^0_\alpha(C_\Sigma;
\Lambda):=\inf_{\scriptstyle u\in C^2_c(C_\Sigma;
\Lambda)\atop \scriptstyle u\ne 0}\frac{\displaystyle\int_{C_\Sigma}|x|^{\alpha}|\Delta u|^2\,dx}{\displaystyle\int_{C_\Sigma}|x|^{\alpha-2}|\nabla u|^2\,dx}\quad \text{and} \quad \mu_\alpha^0(C_\Sigma):=\mu_\alpha^0(C_\Sigma;\Lambda_\Sigma)\,.
\end{equation*}
We prove:

\begin{Theorem}\label{dir>}
If  $\Sigma$ is a $C^2$ domain of $S^{N-1}$ and $\Sigma\ne S^{N-1}$, then $\mu^0_\alpha(C_\Sigma;\Lambda)\ge  \mu_\alpha(C_\Sigma;\Lambda)$ and $\mu^0_\alpha(C_\Sigma;\Lambda)>0$ for every $\alpha\in\R$ and for every nonempty $\Lambda\subseteq\Lambda_\Sigma$. 
\end{Theorem}

\begin{proof} The large inequality is trivial, as $C^2_c(C_\Sigma;
\Lambda)\subset C^2_0(C_\Sigma;
\Lambda)$. Let us prove that  $\mu^0_\alpha(C_\Sigma;
\Lambda)>0$. 
To this aim, let us introduce the spaces $C_c^2(Z_\Sigma;\Lambda):=C_0^2(Z_\Sigma;\Lambda)\cap C_c^2(Z_\Sigma)$ and $C_c^2(\Sigma;\Lambda):=C_0^2(\Sigma;\Lambda)\cap C_c^2(\Sigma)$, and the corresponding values
\begin{equation}\label{mualpha0SigmaZ}
\mu_\alpha^0(Z_\Sigma;\Lambda):=\inf_{\scriptstyle w\in C^2_c(Z_\Sigma;\Lambda)\atop \scriptstyle w\ne 0}\frac{\displaystyle\int_{Z_\Sigma}|L_\alpha w|^2\,dt\,d\sigma}{\displaystyle\int_{Z_\Sigma}Q_\alpha(w)\,dt\,d\sigma}\,,~~~
\mu_\alpha^0(\Sigma;\Lambda):=\inf_{\scriptstyle \varphi\in C^2_c(\Sigma;\Lambda)\atop \scriptstyle \varphi\ne 0}\frac{\displaystyle\int_{\Sigma}|L_\alpha \varphi|^2\,d\sigma}{\displaystyle\int_{\Sigma}Q_\alpha(\varphi)\,d\sigma}\,.
\end{equation}
Arguing as in the proof of Theorem \ref{3inf}, we can show that
\begin{equation} \label{crucial-bis}
\mu_\alpha^0(Z_\Sigma;\Lambda)=\mu_\alpha^0(\Sigma;\Lambda).
\end{equation}
Assume that $\mu^0_\alpha(C_\Sigma;
\Lambda)=0$. Then there exists a sequence $(u_n)\subset C^2_c(C_\Sigma;
\Lambda)$ such that
\[
\int_{C_\Sigma}|x|^\alpha|\Delta u_n|^2\,dx\to 0 \quad \text{as } n\rightarrow +\infty\quad\text{and}\quad\int_{C_\Sigma}|x|^{\alpha-2}|\nabla u_n|^2\,dx=1~\forall n\in\mathbb{N}.
\]
Let $w_n=T(u_n)$ be the Emden-Fowler transform of $u_n$, according to \eqref{Emden-Fowler}. Then $w_n
\in C^2_c(Z_\Sigma;\Lambda)$ and, by \eqref{first-derivative-w}--\eqref{second-derivative-w}, 
\begin{equation}\label{wnnnn}\begin{split}
&\int_{Z_\Sigma}|L_\alpha w_n|^2\,dt\,d\sigma\to 0\,,\quad\int_{Z_\Sigma}|\nabla_\sigma(w_n)_t|^2\,dt\,d\sigma\to 0\,,\\
&\qquad\qquad\int_{Z_\Sigma}\left[Q_\alpha(w_n)+(w_n)_t|^2\right]\,dt\,d\sigma\to 1
\end{split}
\end{equation}
as $n\to\infty$. Since $(w_n)_t(\cdot,\omega)\in H^1_0(\Sigma)$ for every $\omega\in\Sigma$, by the Poincar\'e inequality
\[
\int_{Z_\Sigma}|\nabla_\sigma(w_n)_t|^2\,dt\,d\sigma\ge\min\Lambda\int_{Z_\Sigma}\left|\left(w_n\right)_t\right|^2\,dt\,d\sigma
\]
with $\min\Lambda>0$ since $\Sigma$ is a proper subdomain of $S^{N-1}$. Then $(w_n)_t\to 0$ in $L^2(Z_\Sigma)$ and consequently, using \eqref{wnnnn}, we obtain that $\mu_\alpha^0(Z_\Sigma;\Lambda)=0$. In view of \eqref{crucial-bis}, also $\mu_\alpha^0(\Sigma;\Lambda)=0$. In fact, 
\[
\mu_\alpha^0(\Sigma;\Lambda)\ge\inf_{\scriptstyle \varphi\in H^2_0(\Sigma)\atop \scriptstyle \varphi\ne 0}\frac{\displaystyle\int_{\Sigma}|L_\alpha \varphi|^2\,d\sigma}{\displaystyle\int_{\Sigma}Q_\alpha(\varphi)\,d\sigma}\,.
\]
Then also 
\begin{equation}\label{H20Sigma}
\inf_{\scriptstyle \varphi\in H^2_0(\Sigma)\atop \scriptstyle \varphi\ne 0}\frac{\displaystyle\int_{\Sigma}|L_\alpha \varphi|^2\,d\sigma}{\displaystyle\int_{\Sigma}Q_\alpha(\varphi)\,d\sigma}=0\,.
\end{equation}
By the compactness of the embedding of $H^2_0(\Sigma)$ into $H^1(\Sigma)$, the infimum in \eqref{H20Sigma} is attained by some $\overline\varphi\in H^2_0(\Sigma)$, $\overline{\varphi}\ne 0$. This means that $\overline{\varphi}$ solves $\Delta_\sigma\overline{\varphi}=\gamma_\alpha\overline{\varphi}$ in $\Sigma$, and $\overline{\varphi}\in H^2_0(\Sigma)$. If $-\gamma_{\alpha} \not \in \Lambda_{\Sigma}$, we immediately get a contradiction. If $-\gamma_{\alpha}  \in \Lambda_{\Sigma}$, then we obtain a contradiction with the unique continuation principle, because $\overline\varphi\in H^2_0(\Sigma)$ and $\Sigma$ is a domain of class $C^2$. 
\end{proof}

 \begin{Remark}(Navier case versus Dirichlet case)
When $\Sigma \ne S^{N-1}$, by Theorem \ref{dir>}, the Hardy-Rellich inequality in the Dirichlet case holds true with a positive constant $\mu^0_\alpha(C_\Sigma;\Lambda)$, even when $-\gamma_\alpha$ is an eigenvalue of the Laplace-Beltrami operator in $H^1_0(\Sigma)$ and the optimal constant in the Navier case, $\mu_\alpha(C_\Sigma;\Lambda)$, vanishes, see Remark \ref{RN}-(i). Clearly, in such cases, 
\begin{equation}
\label{strictt}\mu^0_\alpha(C_\Sigma;\Lambda) > \mu_\alpha(C_\Sigma;\Lambda)\,.
\end{equation} 
We conjecture that \eqref{strictt} holds for all $\alpha \in \R$ and for every nonempty $\Lambda\subseteq\Lambda_\Sigma$. 
This issue was established in the Rellich case in \cite[Theorem 3.1]{CM}, by comparing the optimal constants of related infimum problems in $Z_\Sigma$, after applying the Emden transform. Unfortunately, this argument cannot be fully extended to the Hardy-Rellich case due to additional terms arising when performing the Emden transform.
Indeed, we expect that 
\begin{equation}
\label{equalities0}
\mu^0_\alpha(\Sigma;\Lambda) = \mu_\alpha^0(Z_\Sigma;\Lambda) = \mu^0_\alpha(C_\Sigma;\Lambda)
\end{equation}
where $\mu^0_\alpha(\Sigma;\Lambda)$ and $\mu_\alpha^0(Z_\Sigma;\Lambda)$ are defined in \eqref{mualpha0SigmaZ}. This would be enough to deduce \eqref{strictt}. Indeed, by the same arguments in the proof of \cite[Proposition 2.2]{CM}, one can show that $\mu_\alpha^0(\Sigma;\Lambda)>\mu_\alpha(\Sigma;\Lambda)$, and then, one could use \eqref{mu-alpha-Sigma}. Let us examine \eqref{equalities0}. Arguing as in the first part of the proof of Theorem \ref{3inf}, it is easy to show the first equality. 
Moreover, if we set $w^\varepsilon(t,\sigma):=w(\varepsilon t,\sigma)$ for $\varepsilon>0$, with the notation introduced in the proof of Theorem \ref{HR-cone}, we have that
 \[\begin{split}
&\mu_\alpha^0(C_\Sigma;\Lambda)=\inf_{\scriptstyle w\in C^2_c(Z_\Sigma;\Lambda)\atop\scriptstyle w\ne 0}\frac{F(w)}{G(w)}=\inf_{\scriptstyle w\in C^2_c(Z_\Sigma;\Lambda)\atop\scriptstyle  w\ne 0,\,\varepsilon>0}\frac{F(w^\varepsilon)}{G(w^\varepsilon)}\le\inf_{\scriptstyle w\in C^2_c(Z_\Sigma;\Lambda)\atop\scriptstyle w\ne 0}\lim_{\varepsilon\to 0}\frac{F(w^\varepsilon)}{G(w^\varepsilon)}\\
&=\inf_{\scriptstyle w\in C^2_c(Z_\Sigma;\Lambda)\atop\scriptstyle w\ne 0}\lim_{\varepsilon\to 0}\frac{\displaystyle\int_{Z_\Sigma}\left(|L_\alpha w|^2+\varepsilon^4|w_{tt}|^2+2\varepsilon^2|\nabla_\sigma w_t|^2+2\varepsilon^2\overline\gamma_\alpha|w_t|^2\right)\,dt\,d\sigma}{\displaystyle\int_{Z_\Sigma}\left(Q_\alpha(w)+\varepsilon^2|w_t|^2\right)\,dt\,d\sigma}\,.
\end{split}
\]
Hence, $ \mu^0_\alpha(C_\Sigma;\Lambda) \leq \mu_\alpha^0(Z_\Sigma;\Lambda)$. However, proving that also the converse inequality holds appears to be a difficult task due to the additional term in the denominator involving $w_t$, which is absent in the Rellich case.
\end{Remark}

\section{Improved Hardy-Rellich inequalities on cone-like domains}\label{S5}

In this section we study Hardy-Rellich inequalities with remainder terms in domains of the form
\begin{equation}\label{Omega}
\Omega=\{x\in C_\Sigma~:~|x|<1\}\quad\text{or}\quad \Omega=\{x\in C_\Sigma~:~|x|>1\}
\end{equation}
where $C_\Sigma$ is a cone corresponding to a $C^2$ domain $\Sigma$ in $S^{N-1}$, with $N\ge 2$. The case $\Sigma=S^{N-1}$ is also allowed and for this choice we have
\begin{equation*}
\Omega=\{x\in\R^N~:~0<|x|<1\}\quad\text{or}\quad \Omega=\{x\in \R^N~:~|x|>1\}.
\end{equation*}
We are going to obtain different kinds of inequalities, depending on the boundary conditions, often with optimal constants also for the remainder terms. 

The main tool is again the Emden-Fowler transform, defined by \eqref{Emden-Fowler}, which provides an isomorfism between spaces of functions in $C^2(\overline\Omega)$ satisfying some boundary conditions, and corresponding spaces of functions in $C^2(\overline{Z})$ where 
\[
Z=\R_+\times\Sigma\quad\text{or}\quad Z=\R_-\times\Sigma
\]
($\R_+=(0,\infty)$ and $\R_-=(-\infty,0)$), if $\Omega$ is the first or the second domain in \eqref{Omega}. Then, a suitable exploitation of 1-dimensional Hardy and Rellich inequalities:
\begin{equation}
\label{R1d}
\inf_{\scriptstyle\psi\in C^1_c(\R_+)\atop\scriptstyle \psi\ne 0}\frac{\displaystyle\int_0^\infty|\psi'|^2\,dt}{\displaystyle\int_0^\infty t^{-2}|\psi|^2\,dt}=\frac1{4} \quad \text{and}\quad \inf_{\scriptstyle\psi\in C^2_c(\R_+)\atop\scriptstyle \psi\ne 0}\frac{\displaystyle\int_0^\infty|\psi''|^2\,dt}{\displaystyle\int_0^\infty t^{-4}|\psi|^2\,dt}=\frac9{16}
\end{equation}
 allows to derive remainder terms involving Leray-type weights. 
 
For future use, we first note that if $u\colon\overline\Omega\to\R$ is a $C^2$ mapping vanishing on $\partial\Omega$ and with support far from $0$ in the first case of \eqref{Omega} or with compact support in the second case of \eqref{Omega}, then $w=T(u)\colon\overline{Z}\to\R$ is a $C^2$ mapping vanishing on $\partial Z$ and with compact support. Moreover, identities \eqref{w2}--\eqref{second-derivative-w} (with $\Omega$ and $Z$ instead of $C_\Sigma$ and $Z_\Sigma$, respectively) still hold, as well as the following ones:
\begin{gather}
\label{log-nabla}
\int_{\Omega}|x|^{\alpha-2}\left|\log|x|\right|^{-2}|\nabla u|^2\,dx=\int_{Z}t^{-2}\left(|\nabla_\sigma w|^2+|w_t+\widehat\gamma_\alpha w|^2\right)\,dt\,{\dsigma}
\\
\label{log-nabla-sigma}
\int_{\Omega}|x|^{\alpha-4}\left|\log|x|\right|^{-2}|\nabla_\sigma u|^2\,dx=\int_{Z}t^{-2}|\nabla_\sigma w|^2\,dt\,{\dsigma}\\
\label{log-u2}
\int_{\Omega}|x|^{\alpha-4}\left|\log|x|\right|^{-\beta}|u|^2\,dx=\int_{Z}t^{-\beta}|w|^2\,dt\,{\dsigma}
\end{gather}
where  $\alpha \in \R$ and $\beta=2$ or $\beta=4$.

\subsection{The Navier-Dirichlet case}

For $\Omega=\{x\in C_\Sigma~:~|x|<1\}$, let us introduce the space
\[
C^2_{ND}(\Omega):=\{u\in C^2(\overline\Omega)~:~u|_{\partial\Omega}=0,~u=0\text{ in neighborhoods of $0$ and $S^{N-1}\cap \overline\Omega$}\}.
\]
In the case $\Omega=\{x\in C_\Sigma~:~|x|>1\}$, we define
\[
C^2_{ND}(\Omega):=\{u\in C_c^2(\overline\Omega)~:~u|_{\partial\Omega}=0,~u=0\text{ in a neighborhood of $S^{N-1}\cap \overline\Omega$}\}.
\]
The boundary conditions expressed in the above definitions of $C^2_{ND}(\Omega)$ correspond to the so-called Navier-Dirichlet case for domains $\Omega$ of the form \eqref{Omega}. Finally, for every $\Lambda$ nonempty subset of $\Lambda_\Sigma$, we will denote by $C^2_{ND}(\Omega;\Lambda)$ the class of functions in $C^2_{ND}(\Omega)$ satisfying the orthogonality condition \eqref{ortogonale} for all $0<r<1$, respectively, for all $r>1$, in the two cases.

We show the following version of Hardy-Rellich inequality with remainder terms: 
 
\begin{Theorem}\label{ND-case}
Let $\alpha\in\R$ and let $\Omega$ be as in \eqref{Omega}. Furthermore, let $\Lambda$ be a nonempty subset of $\Lambda_\Sigma$.  Then, for every $u\in C^2_{ND}(\Omega;\Lambda)$ one has
\begin{equation}\label{ND}
\begin{split}
\int_{\Omega}   |x|^\alpha  |\Delta u|^2\,dx&\ge  \mu_\alpha(\Sigma;\Lambda)\int_{\Omega}|x|^{\alpha-2}|\nabla u|^2\,dx+\frac14\int_{\Omega}|x|^{\alpha-2}\left|\log|x|\right|^{-2}|\nabla u|^2\,dx\\
&\quad +\frac{\lambda_{\Lambda}(\alpha-2)^2}{4(\lambda_{\Lambda}+\widehat\gamma_\alpha^2)}\int_{\Omega}|x|^{\alpha-4}\left|\log|x|\right|^{-2}|u|^2\,dx
\end{split}
\end{equation}
where $\mu_\alpha(\Sigma;\Lambda)$ is given by \eqref{mu-alpha-Sigma},  $\lambda_{\Lambda}:=\min\Lambda$, and $\widehat\gamma_\alpha$ is defined in \eqref{ga}.
Moreover, if $A_0,A_1,A_2\in\R$ are such that
\begin{equation}\label{ND-A}
\begin{split}
\int_{\Omega}   |x|^\alpha  |\Delta u|^2\,dx&\ge  A_0\int_{\Omega}|x|^{\alpha-2}|\nabla u|^2\,dx +A_1\int_{\Omega}|x|^{\alpha-2}\left|\log|x|\right|^{-2}|\nabla u|^2\,dx\\
&\quad +A_2\int_{\Omega}|x|^{\alpha-4}\left|\log|x|\right|^{-2}|u|^2\,dx
\end{split}
\end{equation}
for every $u\in C^2_{ND}(\Omega;\Lambda)$, then 
\[
A_0\le\mu_\alpha(\Sigma;\Lambda)\,,\quad A_1\le \frac14\,,\text{\quad and\quad} A_2\le \frac{\lambda(\alpha-2)^2}{4(\lambda+\widehat\gamma_\alpha^2)}\,,
\] 
where $\lambda \in \Lambda$ is such that $\mu_\alpha(\Sigma;\Lambda)=\frac{(\lambda+\gamma_\alpha)^2}{\lambda+\widehat\gamma_\alpha^2}$. 
\end{Theorem}

In other words, the coefficients $\mu_\alpha(\Sigma;\Lambda)$ and $\frac14$ in front of the first and the second integral at the right hand side in \eqref{ND} are optimal, whereas the third coefficient is so whenever $\mu_\alpha(\Sigma;\Lambda)$ is achieved by $\lambda_{\Lambda}$ (see Remark \ref{R:principal}). 

\begin{proof}
Let us discuss the case $\Omega=\{x\in C_\Sigma~:~|x|<1\}$ and let us start with some preliminary remarks. The Emden-Fowler transform $u\mapsto T(u)=w$ given by \eqref{Emden-Fowler} defines an isomorphism between the spaces $C^2_{ND}(\Omega)$ and 
\[
C^2_{ND}(Z):=\{w\in C_c^2(\overline{Z})~:~w|_{\partial Z}=0,~w=0\text{ in a neighborhood of }\{0\}\times\Sigma\}.
\]
Furthermore, the class $C^2_{ND}(\Omega;\Lambda)$ is transformed into the class $C^2_{ND}(Z;\Lambda)$ of functions in $C^2_{ND}(Z)$ satisfying the orthogonality condition \eqref{ortogonaleZ} for all $t>0$.
In particular, for $w\in C^2_{ND}(Z;\Lambda)$, we have that $w_t(\cdot,\omega)+\widehat\gamma_\alpha w(\cdot,\omega)\in H^1_0(\R_+)$ for every 
$\omega\in\Sigma$ and then, by the 1-dimensional Hardy inequality \eqref{R1d},
\begin{equation}
\label{H-ww'}
\int_Z\left(|w_{tt}|^2+\widehat\gamma_\alpha^2|w_t|^2\right)\,dt\,{\dsigma}\ge \frac14 \int_Zt^{-2}|w_t+\widehat\gamma_\alpha w|^2\,dt\,{\dsigma}\,.
\end{equation}
Similarly, since $|\nabla_\sigma w(\cdot,\omega)|\in H^1_0(\R_+)$ for every $\omega\in\Sigma$, we also have
\begin{equation}
\label{H-nabla-w}
\int_Z|\nabla_\sigma w_t|^2\,dt\,{\dsigma}\ge \frac14 \int_Zt^{-2}|\nabla_\sigma w|^2\,dt\,{\dsigma}\,.
\end{equation}
Moreover, since $w=0$ on $\partial Z$, we have that $w_t(t,\cdot)\in H^1_0(\Sigma)$ for every $t>0$ and then, in view of the orthogonality condition \eqref{ortogonaleZ},
\begin{equation}
\label{P-w}
\int_Z|\nabla_\sigma w_t|^2\,dt\,{\dsigma}\ge\lambda_\Lambda \int_Z|w_t|^2\,dt\,{\dsigma}\,.
\end{equation}
Finally, we observe that since the odd extension of $w$ with respect to $t$ belongs to $H^2 \cap H^1_0(Z_\Sigma)$, by Theorem \ref{3inf} (which can be exploited in view of the density of $C^2_0(Z_\Sigma)$ in $H^2 \cap H^1_0(Z_\Sigma)$), we have
\begin{equation}\label{halfcilinder}
\int_{Z}|L_\alpha w|^2\,dt\,{\dsigma} \ge \mu_\alpha(\Sigma;\Lambda)\int_{Z}Q_\alpha(w)\,dt\,{\dsigma}\,.
\end{equation}
Now, let us prove the first part of the Theorem. For every $u\in C^2_{ND}(\Omega;\Lambda)$, we have that $w=T(u)\in C^2_{ND}(Z;\Lambda)$ and, using \eqref{first-derivative-w}, \eqref{second-derivative-w}, \eqref{log-nabla} and the inequalities \eqref{H-ww'}--\eqref{halfcilinder}, we estimate
\[
\begin{split}
\int_\Omega&|x|^\alpha|\Delta u|^2\,dx-\mu_\alpha(\Sigma;\Lambda)\int_\Omega|x|^{\alpha-2}|\nabla u|^2\,dx-\frac14\int_{\Omega}|x|^{\alpha-2}\left|\log|x|\right|^{-2}|\nabla u|^2\,dx\\
&=\int_{Z}\left[|L_\alpha w|^2+|w_{tt}|^2+2|\nabla_\sigma w_t|^2+2\overline\gamma_\alpha|w_t|^2\right]\,dt\,{\dsigma}\\
&-\mu_\alpha(\Sigma;\Lambda)\int_{Z}\left[Q_\alpha(w)+|w_t|^2\right]\,dt\,{\dsigma}-\frac{1}{4} \int_{Z}t^{-2}\left(|\nabla_\sigma w|^2+|w_t+\widehat\gamma_\alpha w|^2\right)\,dt\,{\dsigma} \\
&\geq \left(\lambda_\Lambda+2\overline\gamma_\alpha-\mu_\alpha(\Sigma;\Lambda)-\widehat\gamma_\alpha^2\right)\int_{Z_\Sigma^+}|w_t|^2\,dt\,{\dsigma}.
\end{split}\]
In particular, for future reference, we point out that in the above we have used both \eqref{H-nabla-w} and \eqref{P-w} to estimate
\begin{equation}\label{difference}
2\int_Z|\nabla_\sigma w_t|^2\,dt\,{\dsigma}\ge \frac14\int_Zt^{-2}|\nabla_\sigma w|^2\,dt\,{\dsigma}+\lambda_\Lambda\int_Z|w_t|^2\,dt\,{\dsigma}\,.
\end{equation}
Then, recalling that $\mu_\alpha(\Sigma;\Lambda)\le \frac{(\lambda_\Lambda+\gamma_\alpha)^2}{\lambda_\Lambda+\widehat\gamma_\alpha^2}$, we estimate
\[
\lambda_\Lambda+2\overline\gamma_\alpha-\mu_\alpha(\Sigma;\Lambda)-\widehat\gamma_\alpha^2\ge\frac{\lambda_\Lambda(\alpha-2)^2}{\widehat\gamma_\alpha^2+\lambda_\Lambda}
\]
and using the 1-dimensional Hardy inequality and \eqref{log-u2} with $\beta=2$, we obtain \eqref{ND}. \\
Let us now prove the second part of the theorem. For every $w\in C^2_{ND}(Z;\Lambda)$ let $u=T^{-1}(w)\in C^2_{ND}(\Omega;\Lambda)$, where $T$ is the Emden-Fowler transform \eqref{Emden-Fowler}. Using \eqref{first-derivative-w}, \eqref{second-derivative-w}, \eqref{log-nabla}, \eqref{log-u2} and the assumption \eqref{ND-A}, we obtain
\[\begin{split}
0&\le 
\int_{Z}\left[|L_\alpha w|^2+|w_{tt}|^2+2|\nabla_\sigma w_t|^2+2\overline\gamma_\alpha|w_t|^2\right]-A_0\int_{Z}\left[Q_\alpha(w)+|w_t|^2\right]\\
&\quad-A_1\int_{Z}t^{-2}\left(|\nabla_\sigma w|^2+|w_t+\widehat\gamma_\alpha w|^2\right)-A_2\int_Z t^{-2}|w|^2.
\end{split}\]
The same inequality holds true with $w^\varepsilon(t,\omega)=w(\varepsilon t,\omega)$ instead of $w$, for every $\varepsilon>0$. This yields that
\begin{equation}
\label{ND-eps}
\begin{split}
0&\le \frac1\varepsilon\int_{Z}|L_\alpha w|^2+\varepsilon^3\int_{Z}|w_{tt}|^2+2\varepsilon\int_{Z}\left[|\nabla_\sigma w_t|^2+\overline\gamma_\alpha|w_t|^2\right]\\
&\quad-\frac{A_0}\varepsilon\int_{Z}Q_\alpha(w)-\varepsilon A_0\int_{Z}|w_t|^2-\varepsilon A_1\int_{Z}t^{-2}\left[|\nabla_\sigma w|^2+\widehat\gamma_\alpha^2|w|^2\right]\\
&\quad-\varepsilon^3 A_1\int_{Z}t^{-2}|w_t|^2-2\varepsilon^2 A_1\widehat\gamma_\alpha\int_{Z}t^{-2}ww_t-A_2\varepsilon\int_{Z}t^{-2}|w|^2.
\end{split}
\end{equation}
Multiplying \eqref{ND-eps} by $\varepsilon>0$ and letting $\varepsilon\to 0$, we obtain
\begin{equation*}
A_0\int_{Z}Q_\alpha(w)\,dt\,{\dsigma}\le \int_{Z}|L_\alpha w|^2\,dt\,{\dsigma}\quad\forall w\in C^2_{ND}(Z;\Lambda).
\end{equation*}
By Theorem \ref{3inf}, $A_0\le\mu_\alpha(\Sigma;\Lambda)$. 
Multiplying \eqref{ND-eps} by $\varepsilon^{-3}>0$ and letting $\varepsilon\to\infty$, we deduce that
\[
A_1\int_{Z}t^{-2}|w_t|^2\,dt\,{\dsigma}\le \int_{Z}|w_{tt}|^2\,dt\,{\dsigma}\quad\forall w\in C^2_{ND}(Z;\Lambda).
\]
In particular, if in the above we take $w(t,\omega)=\psi(t)\varphi(\omega)$ with $\psi\in C_c^2(\R_+)$, $\psi\ne 0$, and $\varphi$ eigenfunction corresponding to $\lambda\in \Lambda $, recalling \eqref{R1d} we infer that
\[
A_1\le \inf_{\scriptstyle\psi\in C^2_c(\R_+)\atop\scriptstyle\psi\ne 0}\frac{\displaystyle\int_0^\infty|\psi''|^2\,dt}{\displaystyle\int_0^\infty t^{-2}|\psi'|^2\,dt}=\frac14.
\]
Instead, if in \eqref{ND-eps} we assume that $w(t,\omega)$ as given above with $\varphi$ eigenfunction corresponding to $\lambda\in \Lambda $ such that $\mu_\alpha(\Sigma;\Lambda)={(\lambda+\gamma_\alpha)^2}/{(\lambda+\widehat\gamma_\alpha^2)}$, 
and dividing by $\varepsilon>0$ we get
\[\begin{split}
A_2\int_0^\infty &t^{-2}|\psi|^2\,dt\le\frac{(\lambda+\gamma_\alpha)^2-A_0(\lambda+\widehat\gamma_\alpha^2)}{\varepsilon^2}\int_0^\infty|\psi|^2\,dt\\
&\quad+(2\lambda+2\overline\gamma_\alpha-A_0)\int_0^\infty|\psi'|^2\,dt-A_1(\lambda+\widehat\gamma_\alpha^2)\int_0^\infty t^{-2}|\psi|^2\,dt\\
&\quad-2\varepsilon A_1\widehat\gamma_\alpha\int_0^\infty t^{-2}\psi\psi'\,dt+\epsilon^2\left(\int_0^\infty|\psi''|^2\,dt-A_1\int_0^\infty t^{-2}|\psi'|^2\,dt\right).
\end{split}\]
We can set
\begin{equation}
\label{NDA0}
A_0=\frac{(\lambda+\gamma_\alpha)^2}{\lambda+\widehat\gamma_\alpha^2}=\mu_\alpha(\Sigma;\Lambda)
\end{equation}
and letting $\varepsilon\to 0$, we arrive to estimate
\[
A_2\int_0^\infty t^{-2}|\psi|^2\,dt\le(2\lambda+2\overline\gamma_\alpha-\mu_\alpha(\Sigma;\Lambda))\int_0^\infty|\psi'|^2\,dt-A_1(\lambda+\widehat\gamma_\alpha^2)\int_0^\infty t^{-2}|\psi|^2\,dt
\]
for every $\psi\in C^2_c(\R_+)$. We point out that, by \eqref{NDA0},
\begin{equation}
\label{kappapositive}
2\lambda+2\overline\gamma_\alpha-\mu_\alpha(\Sigma;\Lambda)=p_\alpha(\lambda)\ge 0
\end{equation} (see \eqref{p-alpha}). Thus, by the 1-dimensional Hardy inequality, selecting $A_1=\frac14$ and using the expression of $\mu_\alpha(\Sigma;\Lambda)$ as in \eqref{NDA0}, we finally obtain
\[
A_2\le \frac14\left[2\lambda+2\overline\gamma_\alpha-\frac{(\lambda+\gamma_\alpha)^2}{\lambda+\widehat\gamma_\alpha^2}\right]-\frac{\lambda+\widehat\gamma_\alpha^2}4=\frac{\lambda(\alpha-4)^2}{4(\widehat\gamma_\alpha^2+\lambda)}
\]
by the expressions of $\gamma_\alpha$, $\overline\gamma_\alpha$ and $\widehat\gamma_\alpha$ provided in \eqref{ga}. The case $\Omega=\{x\in C_\Sigma~:~|x|>1\}$ can be studied exactly in the same way, with the only difference that $Z=\R_-\times\Sigma$. 
\end{proof} 

Using the same technique applied in the previous proof, we can prove another version of Hardy-Rellich inequality in $C^2_{ND}(\Omega;\Lambda)$ with remainder terms which involve only $L^2$-norms of $u$ with logarithmic weights, as shown below.

\begin{Theorem}\label{52}
Let $\alpha\in\R$ and let $\Omega$ be as in \eqref{Omega}. Furthermore, let $\Lambda$ be a nonempty subset of $\Lambda_\Sigma$. Then for every $u\in C^2_{ND}(\Omega;\Lambda)$ one has
\begin{equation}\label{ND2}
\begin{split}
\int_{\Omega}|x|^\alpha|\Delta u|^2 dx&\ge  \mu_\alpha(\Sigma;\Lambda)\int_{\Omega}|x|^{\alpha-2}|\nabla u|^2 dx+\kappa_{\alpha}(\Sigma;\Lambda)\int_{\Omega}|x|^{\alpha-4}\left|\log|x|\right|^{-2}|u|^2 dx\\
&\quad +\frac9{16}\int_{\Omega}|x|^{\alpha-4}\left|\log|x|\right|^{-4}|u|^2 dx
\end{split}
\end{equation}
where 
\begin{equation}
\label{kalpha}
\kappa_\alpha(\Sigma;\Lambda):=\frac{2\overline\gamma_\alpha+2\lambda_\Lambda-\mu_\alpha(\Sigma;\Lambda)}4\ge 0\,,
\end{equation} 
$\mu_\alpha(\Sigma;\Lambda)$ is given by \eqref{mu-alpha-Sigma}, $\lambda_{\Lambda}:=\min\Lambda$, and $\overline\gamma_\alpha$ is defined in \eqref{ga}.
Moreover, if $A_0,A_1,A_2\in\R$ are such that
\[\begin{split}
\int_{\Omega}   |x|^\alpha  |\Delta u|^2\,dx&\ge  A_0\int_{\Omega}|x|^{\alpha-2}|\nabla u|^2\,dx +A_1\int_{\Omega}|x|^{\alpha-4}\left|\log|x|\right|^{-2}|u|^2\,dx\\
&\quad +A_2\int_{\Omega}|x|^{\alpha-4}\left|\log|x|\right|^{-4}|u|^2\,dx
\end{split}
\]
for every $u\in C^2_{ND}(\Omega;\Lambda)$, then $A_0\le\mu_\alpha(\Sigma;\Lambda)$, $A_2\le \frac9{16}$, and $A_1\le(2\lambda+2\overline\gamma_\alpha-\mu_\alpha(\Sigma;\Lambda))/4$, 
where $\lambda \in \Lambda$ is such that $\mu_\alpha(\Sigma;\Lambda)=\frac{(\lambda+\gamma_\alpha)^2}{\lambda+\widehat\gamma_\alpha^2}$. 
\end{Theorem}

Hence, the coefficients $\mu_\alpha(\Sigma;\Lambda)$ and $9/16$ at the right hand side in \eqref{ND2} are optimal, whereas the  coefficient $\kappa_\alpha(\Sigma;\Lambda)$ is so whenever $\mu_\alpha(\Sigma;\Lambda)$ is realized in correspondence of $\lambda_{\Lambda}$ (see Remark \ref{R:principal}).

\begin{proof}
Firstly, let us observe that $\kappa_\alpha(\Sigma;\Lambda)\ge 0$ because of \eqref{kappapositive}. Now, the first part of the theorem can be shown with the same procedure used in the proof of Theorem \ref{ND-case}. In particular, it is more convenient to estimate just
\begin{equation}\label{reduced-difference}
\begin{split}
\int_\Omega|x|^\alpha|\Delta u|^2\,dx\,-\,&\mu_\alpha(\Sigma;\Lambda)\int_\Omega|x|^{\alpha-2}|\nabla u|^2\,dx\\
&=\int_{Z}\left[|L_\alpha w|^2+|w_{tt}|^2+2|\nabla_\sigma w_t|^2+2\overline\gamma_\alpha|w_t|^2\right]dt\,{\dsigma}\\
&\quad-\mu_\alpha(\Sigma;\Lambda)\int_{Z}\left[Q_\alpha(w)+|w_t|^2\right]dt\,{\dsigma}\,.
\end{split}
\end{equation}
Then, one uses \eqref{P-w},\eqref{halfcilinder}, the 1-dimensional Hardy and Rellich inequalities \eqref{R1d} and the identity \eqref{log-u2} with $\beta=2$ and $\beta=4$ and plainly arrives to \eqref{ND2}. Concerning the second part of the theorem, we follow again the same strategy of the proof of theorem \ref{ND-case}. The only difference is that the estimate of $A_2$ is obtained by means of \eqref{R1d}. 
\end{proof}

\subsection{The Dirichlet case in the punctured unit ball and in the exterior of the unit ball}

When $\Omega=\{x\in\R^N~:~0<|x|<1\}$ (or $\Omega=\{x\in\R^N~:~|x|>1\}$), we have that $C^2_{ND}(\Omega)=C^2_c(\Omega)$. Therefore, for every $\Lambda$ nonempty subset of $\Lambda_\Sigma$, Theorem \ref{ND-case} still holds in the space $C^2_c(\Omega;\Lambda)$ of functions in $C^2_c(\Omega)$ satisfying the orthogonality condition \eqref{ortogonale} for all $0<r<1$ (or for all $r>1$). However, when $\Lambda=\Lambda_\Sigma$ the last term in \eqref{ND} disappears, since the principal eigenvalue in $\Sigma=S^{N-1}$ is zero. Theorem \ref{punctured ball} below provides a different improved version of the Hardy-Rellich inequality in $C^2_c(\Omega;\Lambda)$, with an additional term involving the spherical gradient. 

\begin{Theorem}\label{punctured ball}
Let $\alpha\in\R$ and let $\Omega=\{x\in\R^N~:~0<|x|<1\}$ or $\Omega=\{x\in\R^N~:~|x|>1\}$ (with $N\ge 2$).  Furthermore, let $\Lambda$ be a nonempty subset of $\Lambda_\Sigma$.  Then, for every $u\in C^2_c(\Omega;\Lambda)$ one has
\begin{equation}\label{Nav_Dir_ineq_tris}
\begin{split}
\int_{\Omega}&|x|^\alpha|\Delta u|^2\,dx\ge \mu_\alpha(S^{N-1};\Lambda) \int_{\Omega}|x|^{\alpha-2}|\nabla u|^2 dx+\frac14\int_{\Omega}|x|^{\alpha-2}\left|\log|x|\right|^{-2}|\nabla u|^2 dx\\
&\quad+\frac{1}{4} \int_{\Omega} |x|^{\alpha-4} \left|\log|x|\right|^{-2} |\nabla_{\sigma} u|^2 dx+K_\alpha(S^{N-1};\Lambda) \int_{\Omega}|x|^{\alpha-4}\left|\log|x|\right|^{-2}|u|^2dx
\end{split}
\end{equation}
where $\mu_\alpha(S^{N-1};\Lambda)$ is given by \eqref{mu-alpha-Sigma}, $K_\alpha(S^{N-1};\Lambda) :=(2\overline\gamma_\alpha-\widehat\gamma_\alpha^2-\mu_\alpha(S^{N-1};\Lambda) )/4$, and $\overline\gamma_\alpha$ and $\widehat\gamma_\alpha$ are defined in \eqref{ga}. 
Moreover, if $A_0,A_1,A_2\in\R$ are such that
\begin{equation}\label{ND-AA}
\begin{split}
\int_{\Omega}   |x|^\alpha  |\Delta u|^2\,dx&\ge  A_0\int_{\Omega}|x|^{\alpha-2}|\nabla u|^2\,dx +A_1\int_{\Omega}|x|^{\alpha-4}\left|\log|x|\right|^{-2}|u|^2\,dx\\
&\quad +A_2\int_{\Omega} |x|^{\alpha-4} \left|\log|x|\right|^{-2} |\nabla_{\sigma} u|^2\,dx
\end{split}
\end{equation}
for every $u\in C^2_{c}(\Omega)$, then $A_0\le\mu_\alpha(S^{N-1};\Lambda) $ and $A_1\le \frac14$. 
 Furthermore, if $\Lambda=\Lambda_{\Sigma}$, then $A_2\le \inf_{\psi\in C^2_c(\R_+)\setminus\{0\}}R(\psi)$, where
\begin{equation}\label{R(psi)}
R(\psi)=\frac{2\left[\int_0^\infty|\psi|^2\,dt\right]^{\frac12}\left[\int_0^\infty|\psi''|^2\,dt-\frac14\int_0^\infty t^{-2}|\psi'|^2\,dt\right]^{\frac12}}{\int_0^\infty t^{-2}|\psi|^2\,dt}+2\,\frac{\int_0^\infty|\psi'|^2\,dt}{\int_0^\infty t^{-2}|\psi|^2\,dt}-\frac14\,.
\end{equation}
\end{Theorem}

\begin{Remark} 
Recalling \eqref{ga} and \eqref{HR-RN}, $K_\alpha(S^{N-1};\Lambda_\Sigma)=\left(\frac{N-\alpha}{2}\right)^2-\mu_\alpha(S^{N-1})\geq 0$. Hence, by \eqref{HR-RN}, $K_\alpha(S^{N-1};\Lambda_\Sigma)= 0$ if and only if the minimum in \eqref{mu-alpha-Sigma} is achieved by the principal eigenvalue (see Remark \ref{R:principal}). 
\end{Remark}

\begin{Remark}\label{Dirichlet-rem_terms} In \cite{GPS} the authors prove a family of Hardy-Rellich-type inequality in the ball $B_R=\{x\in\R^N~:~|x|<R\}$ with logarithmic refinements. In particular, when $R=1$, \cite[Theorem 2.3, case $N=1$]{GPS} yields the following inequality:
\begin{equation}\label{GPS_ineq}
\begin{split}
&\int_{B_1}|x|^\alpha|\Delta u|^2\,dx\ge \mu_\alpha (S^{N-1})\int_{B_1}|x|^{\alpha-2}|\nabla u|^2\,dx\\
&+\frac14\int_{B_1}|x|^{\alpha-2}\left|\log|x|\right|^{-2}|\nabla u|^2 \,dx+\frac{1}{4} \int_{B_1} |x|^{\alpha-4} \left|\log|x|\right|^{-2} |\nabla_{\sigma} u|^2\,dx\,,
\end{split}
\end{equation}
for all $u\in C^\infty_c(B_1\setminus\{0\})$, which overlaps with \eqref{Nav_Dir_ineq_tris} (with $\Lambda=\Lambda_{\Sigma}$) except for the last term which appears new. This is related to the different techniques applied in the proofs. Furthermore, \eqref{GPS_ineq},  without the term involving the spherical gradient and stated for test functions in $C^\infty_c(B_1)$, also follows from \cite[Theorem 1.8]{TZ}, under some restrictions on the parameter $\alpha$. In \cite{TZ} the optimality of the constant $\frac{1}{4}$, in front of the remainder term with the gradient and logarithmic weight, was shown in the energy space $H^2_0(B_1)$. As far we are aware, this optimality issue in the functional space $C^2_c(B_1\setminus\{0\})$, with no restrictions on $\alpha$, was not known before, see also \cite[Remark 2.4-(iii)]{GPS}. The optimality of the constant $\mu_\alpha (S^{N-1})$ was instead proved in \cite{GPS} and \cite{GPS2}. We remark that the results in \cite{TZ} hold for any bounded domain containing the origin and  Sobolev-type remainder terms were also provided, see also \cite{BT}.
\end{Remark}

\begin{Remark}
Here we provide a rough estimate of the infimum of the functional $R(\psi)$ defined in \eqref{R(psi)} over the class of non zero functions $\psi\in C^2_c(\R_+)$. By density, we can take $\psi\in H^2_0(\R_+)=\{\psi\in H^2(\R_+)~:~\psi(0)=\psi'(0)=0\}$. 

For $\eps>0$ small, let 
\(
\psi_\eps(t)=t^{\frac32+\eps}e^{-t}
\).
Then $\psi_\eps\in H^2_0(\R^+)$ and, after some computation, one finds
\[
R(\psi_\eps)=
\sqrt{\frac{33}{2}}+\frac{5}4+o(1) \quad \text{with }o(1)\to 0 \text{ as } \eps\to 0\,.
\]
Hence, the coefficient $A_2$ in \eqref{ND-AA} satisfies $A_2\le \sqrt{\frac{33}{2}}+\frac{5}4\approx 5.3$.
\end{Remark}

\begin{proof}
The first part of the theorem can be proved with the same procedure used in the proof of Theorem \ref{ND-case}, replacing the estimate \eqref{difference} with \eqref{H-nabla-w}. In this way, taking into account of \eqref{log-nabla-sigma}, we plainly arrive to
\[\begin{split}
&\int_\Omega|x|^\alpha|\Delta u|^2\,dx-\mu_\alpha(S^{N-1};\Lambda) \int_\Omega|x|^{\alpha-2}|\nabla u|^2\,dx-\frac14\int_{\Omega}|x|^{\alpha-2}\left|\log|x|\right|^{-2}|\nabla u|^2\,dx\\
&=\int_{Z}\left[|L_\alpha w|^2+|w_{tt}|^2+2|\nabla_\sigma w_t|^2+2\overline\gamma_\alpha|w_t|^2\right]\,dt\,{\dsigma}\\
&-\mu_\alpha(S^{N-1};\Lambda) \int_{Z}\left[Q_\alpha(w)+|w_t|^2\right] dt\,{\dsigma}-\frac{1}{4} \int_{Z}t^{-2}\left(|\nabla_\sigma w|^2+|w_t+\widehat\gamma_\alpha w|^2\right)dt\,{\dsigma} \\
&\geq  \frac14 \int_{Z}  t^{-2} |\nabla_\sigma w|^2 dt\,{\dsigma}  + \left(2\overline\gamma_\alpha-\mu_\alpha(S^{N-1};\Lambda) -\widehat\gamma_\alpha^2\right)\int_{Z}|w_t|^2 dt\,{\dsigma}\\
& \geq  \frac{1}{4} \int_{\Omega} {|x|^{\alpha-4}\left|\log|x|\right|^{-2}|\nabla_{\sigma} u|^2} dx + K_\alpha(S^{N-1};\Lambda) \int_{\Omega}|x|^{\alpha-4}\left|\log|x|\right|^{-2}|u|^2 dx\,. 
\end{split}\]
Concerning the second part of the theorem, we follow again the same strategy of the proof of theorem \ref{ND-case} and we arrive to the following estimate:
\begin{equation}
\label{optn}
\begin{split}
0&\le \frac1\varepsilon\int_{Z}|L_\alpha w|^2+\varepsilon^3\int_{Z}|w_{tt}|^2+2\varepsilon\int_{Z}\left[|\nabla_\sigma w_t|^2+\overline\gamma_\alpha|w_t|^2\right]\\
&\quad-\frac{A_0}\varepsilon\int_{Z}Q_\alpha(w)-\varepsilon A_0\int_{Z}|w_t|^2-\varepsilon A_1\int_{Z}t^{-2}\left[|\nabla_\sigma w|^2+\widehat\gamma_\alpha^2|w|^2\right]\\
&\quad-\varepsilon^3 A_1\int_{Z}t^{-2}|w_t|^2-2\varepsilon^2A_1\widehat\gamma_\alpha\int_{Z}t^{-2}ww_t-\varepsilon A_2\int_{Z}t^{-2}|\nabla_\sigma w|^2\,.
\end{split}
\end{equation}
Arguing as in the proof of theorem \ref{ND-case}, we obtain that $A_0\le\mu_\alpha(S^{N-1};\Lambda) $ and $A_1\le\frac14$. The new part in the proof concerns the estimate of $A_2$ when $\Lambda=\Lambda_{\Sigma}$. To this aim, we write \eqref{optn} with $A_0=\mu_\alpha(S^{N-1};\Lambda) $, $A_1=\frac14$ and $w(t,\omega)=\varphi(\omega)\psi(t)$, with $\psi$ as before and $\varphi$ eigenfunction of $-\Delta_\sigma$ in $H^1(S^{N-1})$ corresponding to an eigenvalue $\lambda_j=j(j+N-1)$ with $j\ge 1$. 
Then \eqref{optn} yields
\[
\begin{split}
A_2&\le
\frac1{\varepsilon^2}\frac{\left[(\lambda_j+\gamma_\alpha)^2-\mu_\alpha(\lambda_j+\widehat\gamma_\alpha^2)\right]\int_0^\infty|\psi|^2\,dt}{\lambda_j\int_0^\infty t^{-2}|\psi|^2\,dt}
+\varepsilon^2\frac{\int_0^\infty|\psi''|^2\,dt-\frac14\int_0^\infty t^{-2}|\psi'|^2\,dt}{\lambda_j\int_0^\infty t^{-2}|\psi|^2\,dt}\\
&\quad -\frac{\varepsilon\widehat\gamma_\alpha}{2}\frac{\int_0^\infty t^{-2}\psi\psi'\,dt}{\lambda_j\int_0^\infty t^{-2}|\psi|^2\,dt}+\frac{(2\lambda_j+2\overline\gamma_\alpha-\mu_\alpha)\int_0^\infty|\psi'|^2\,dt}{\lambda_j\int_0^\infty t^{-2}|\psi|^2\,dt}-\frac14\left(1-\frac{\overline\gamma_\alpha}{\lambda_j}\right)
\end{split}\]
for every $\varepsilon>0$ and for every $\psi\in C^2_c(\R_+)$. 
Taking $\varepsilon=\sqrt{\beta\lambda_{j}}$ with $\beta>0$ and letting $j\to\infty$, we obtain
\[
A_2\le \beta\,\frac{\int_0^\infty|\psi|^2\,dt}{\int_0^\infty t^{-2}|\psi|^2\,dt}+\frac1\beta\,\frac{\int_0^\infty|\psi''|^2\,dt-\frac14\int_0^\infty t^{-2}|\psi'|^2\,dt}{\int_0^\infty t^{-2}|\psi|^2\,dt}+2\,\frac{\int_0^\infty|\psi'|^2\,dt}{\int_0^\infty t^{-2}|\psi|^2\,dt}-\frac14\,.
\]
Minimizing with respect to $\beta>0$ we obtain $A_2\le R(\psi)$ 
for every $\psi\in C^2_c(\R_+)$.
\end{proof}

\subsection{The weight-free Dirichlet case in the punctured unit ball and in the exterior of the unit ball}
 Here, we specialize Theorem \ref{52} to the case of the punctured unit ball or the exterior of the unit ball, without the weight in front of the laplacian, i.e., $\alpha=0$. 
 \begin{Corollary}\label{COR522}
Let $\Omega=\{x\in\R^N~:~0<|x|<1\}$ or $\Omega=\{x\in\R^N~:~|x|>1\}$ (with $N\ge 2$). Furthermore, let $\Lambda$ be a nonempty subset of $\Lambda_\Sigma$ and let $I_\Lambda=\{j\in\mathbb{N}~:~\lambda_j=j(j+N-2)\in  \Lambda  \}$. Then, for every $u\in C^2_{c}(\Omega;\Lambda)$ one has
\begin{equation}\label{ND22}
\begin{split}
\int_{\Omega}|\Delta u|^2\,dx&\ge  \mu_0(S^{N-1};\Lambda)\int_{\Omega}\frac{|\nabla u|^2}{|x|^{2}}\,dx+\kappa_{0}(S^{N-1};\Lambda)\int_{\Omega} \frac{|u|^2}{|x|^{4}\log^2|x|}\,dx \\
&\quad +\frac9{16}\int_{\Omega} \frac{|u|^2}{|x|^{4}\log^4|x|}\,dx
\end{split}
\end{equation}
where 
\begin{equation*}
\mu_0(S^{N-1};\Lambda)=\min_{j\in I_\Lambda} \frac{(N(N-4)+4j(j+N-2))^2}{4((N-4)^2+4j(j+N-2))} \,,
\end{equation*} 
with the agreement that the term corresponding to $j=0$ is $N^2/4$, and
\begin{equation*}
\kappa_0(\Lambda;S^{N-1})=\frac{N(N-4)+8+4 j_{\Lambda}(j_{\Lambda}+N-2)-2\mu_0(\Lambda)}{8} \quad \text{with } j_{\Lambda}:=\min I_\Lambda\,.
\end{equation*} 

Moreover, $\mu_{0}(S^{N-1};\Lambda)$ and $\frac{9}{16} $ are optimal for all $N\geq 2$ while $\kappa_0(S^{N-1};\Lambda)$ is optimal if  $j_{\Lambda}$ is the index achieving the minimum in $\mu_0(\Lambda)$.
\end{Corollary}

\begin{Remark}
 When $\Omega$ is the open unit ball in $\mathbb{R}^4$, Adimurthi, Grossi, and Santra in  \cite[Theorem 2.1-(b)]{AGS} prove an inequality similar to \eqref{ND22}. In particular, they show that
\begin{equation}\label{AGS2}\begin{split}
\int_\Omega|\Delta u|^2&\ge 4\int_\Omega\frac{|\nabla u|^2}{|x|^2}-3\int_\Omega\frac{|P_1u|^2}{|x|^4}\\
&\quad+\frac34\int_\Omega\frac{|P_1u|^2}{|x|^4\log^2{\left|x\right|}}+\frac9{32}\int_\Omega\frac{|P_1u|^2}{|x|^4\log^4{\left|x\right|}}\quad\forall u\in H^2_0(\Omega)
\end{split}\end{equation}
where $P_1u$ is the projection of $u$ on the eigenspace  corresponding to the principal eigenvalue of $-\Delta_\sigma$ in $H^1_0(S^3)$, given by $\lambda_1=3$. In fact, when $P_1u=u$ one has
\[
\int_\Omega\frac{|\nabla u|^2}{|x|^2}\ge \int_\Omega\frac{|\nabla_\sigma u|^2}{|x|^2}=3\int_\Omega\frac{|u|^2}{|x|^2}
\]
and \eqref{AGS2} implies that
\begin{equation}\label{AGS3}\begin{split}
\int_\Omega|\Delta (P_1u)|^2&\ge 3\int_\Omega\frac{|\nabla(P_1u)|^2}{|x|^2}+\frac34\int_\Omega\frac{|P_1u|^2}{|x|^4\log^2{\left|x\right|}}+\frac9{32}\int_\Omega\frac{|P_1u|^2}{|x|^4\log^4{\left|x\right|}}
\end{split}\end{equation}
for all $u\in H^2_0(\Omega)$, whereas 
\begin{equation}\label{AGS1}
\int_\Omega|\Delta u|^2\ge 4\int_\Omega\frac{|\nabla u|^2}{|x|^2}
\end{equation}
for all $u\in H^2_0(\Omega)$ such that $P_1u=0$.
We can compare \eqref{AGS3} and \eqref{AGS1} with \eqref{ND22} for $N=4$ and $\Lambda=\{\lambda_1\}$ or $\Lambda=\Lambda_{S^3}\setminus \{\lambda_1\}$, which yields 
\begin{equation}\label{AGS3+}
\int_\Omega|\Delta (P_1u)|^2\ge 3\int_\Omega\frac{|\nabla(P_1u)|^2}{|x|^2}+\frac74\int_\Omega\frac{|P_1u|^2}{|x|^4\log^2{\left|x\right|}}+\frac9{16}\int_\Omega\frac{|P_1u|^2}{|x|^4\log^4{\left|x\right|}}
\end{equation}
for all $u\in H^2_0(\Omega)$ and
\begin{equation}\label{AGS1+}
\int_\Omega|\Delta u|^2\ge 4\int_\Omega\frac{|\nabla u|^2}{|x|^2}+\frac9{16}\int_\Omega\frac{|u|^2}{|x|^2\log^4|x|}
\end{equation}
for all $u\in H^2_0(\Omega)$  such that $P_1u=0$, which improve \eqref{AGS3} and \eqref{AGS1}, respectively. Note that in \eqref{AGS3+} and \eqref{AGS1+} we can take $u\in H^2_0(\Omega)$ because in dimension $N=4$ the weight $|x|^{-2}$ is integrable in $\Omega$ (see Remark \ref{RN}) and $C^2_c(\Omega)$ is dense in $H^2_0(\Omega)$.

\end{Remark}

We conclude the section by sumarizing in a unique statement the improved inequalities provided by Corollary \ref{COR522} and Theorem \ref{punctured ball} in the weight-free case, when no orthogonality condition is imposed, i.e. $\Lambda=\Lambda_\Sigma$.

\begin{Corollary}
Let $\Omega=\{x\in\R^N~:~0<|x|<1\}$ or $\Omega=\{x\in\R^N~:~|x|>1\}$ (with $N\ge 2$). For every $u \in C^2_c(\Omega)$ we have
\[
\int_{\Omega} |\Delta u|^2\,dx\ge \mu_{0} \int_{\Omega}\frac{|\nabla u|^2}{|x|^{2}}\,dx+\kappa_0 \int_{\Omega} \frac{|u|^2}{|x|^{4}\log^2|x|}\,dx+\frac{9}{16} \int_{\Omega} \frac{|u|^2}{|x|^{4}\log^4|x|}\,dx
\]
and also
\[
\int_{\Omega} |\Delta u|^2\,dx\ge \mu_{0} \int_{\Omega}\frac{|\nabla u|^2}{|x|^{2}}\,dx+\frac14 \int_{\Omega} \frac{|\nabla u|^2}{|x|^{2}\log^2|x|}\,dx+\frac{1}{4} \int_{\Omega} \frac{|\nabla_\sigma u|^2}{|x|^{4}\log^2|x|}\,dx
\]
where
\[
\mu_{0}=\mu_{0}(N)=\begin{cases}0& \text{for $N=2$}\\
\frac{25}{36}& \text{for $N=3$}\\ 3&\text{for $N=4$}\\ \frac{N^2}4&\text{for $N\ge 5$}
\end{cases} \quad \text{and} \quad \kappa_0=\kappa_0(N)=\begin{cases}\frac{1}{2}& \text{for $N=2$}\\
\frac{65}{144}& \text{for $N=3$}\\ \frac{1}{4}&\text{for $N=4$}\\ \left(\frac{N-4}4\right)^2&\text{for $N\ge 5$\,.}
\end{cases}
\] 
Furthermore, $\mu_{0}, \frac{9}{16} $ and $\frac{1}{4}$ (in front of the complete gradient) are optimal for all $N\geq 2$ while $\kappa_0$ is optimal for $N\geq 5$.
\end{Corollary}

\subsection{The Dirichlet case in proper cone-like domains}
For every $\Lambda$ nonempty subset of $\Lambda_\Sigma$, we denote by $C^2_{c}(\Omega;\Lambda)$ the set of functions in $C^2_{c}(\Omega)$ satisfying the orthogonality condition \eqref{ortogonale} for all $0<r<1$, respectively, for all $r>1$, in the two cases.

\begin{Theorem}\label{last}
Let $\Sigma$ be a proper subdomain of $S^{N-1}$ of class $C^2$,  let $\Omega$ be as in \eqref{Omega}, and let $\alpha\in\R$. Furthermore, let $\Lambda$ be a nonempty subset of $\Lambda_\Sigma$. Then for every $u\in C^2_c(\Omega;\Lambda)$ one has
\begin{equation}\label{eqlast}
\begin{split}
&\int_{\Omega}  |x|^\alpha|\Delta u|^2\,dx- \mu_\alpha^0(\Sigma;\Lambda) \int_{\Omega}|x|^{\alpha-2}|\nabla u|^2\,dx \\ &\ge \kappa_\alpha^0(\Sigma;\Lambda)\int_{\Omega} |x|^{\alpha-4} \left| \widehat \gamma_\alpha u + x\cdot \nabla u \right|^{2} \,dx+\frac{9}{16} \int_{\Omega} |x|^{\alpha-4} \left|\log|x|\right|^{-4}|u|^2\,dx
\end{split}
\end{equation}
where $\mu_\alpha^0(\Sigma;\Lambda)$ is given by \eqref{crucial-bis}, $\kappa_\alpha^0(\Sigma;\Lambda):=(2\overline\gamma_\alpha +2 \lambda_{\Lambda}-\mu^0_\alpha(\Sigma;\Lambda))/4$, $\lambda_{\Lambda}:=\min\Lambda$, and $\widehat \gamma_\alpha,\overline\gamma_\alpha$ are defined in \eqref{ga}. 
Moreover, if $A_0,A_1\in\R$ are such that
\[\begin{split}
\int_{\Omega}   |x|^\alpha  |\Delta u|^2\,dx&\ge  A_0\int_{\Omega}|x|^{\alpha-2}|\nabla u|^2\,dx +A_1\int_{\Omega}|x|^{\alpha-4}\left|\log|x|\right|^{-4}|u|^2\,dx
\end{split}
\]
for every $u\in C^2_{c}(\Omega)$, then $A_0\le \mu_\alpha^0(\Sigma;\Lambda)$ and $A_1\le \frac9{16}$.
\end{Theorem}

The proof of Theorem \ref{last} is very similar to that of Theorem \ref{52}. In particular, one writes the left hand side of \eqref{eqlast} like \eqref{reduced-difference} and uses \eqref{crucial-bis}, \eqref{P-w}, and the 1-dimensional Rellich inequality \eqref{R1d}.

\begin{Remark}
We point out that we only know that $\mu_\alpha^0(\Sigma;\Lambda)\ge\mu_\alpha^0(C_\Sigma;\Lambda)\ge\mu_\alpha(C_\Sigma;\Lambda)$ (see Section \ref{S3}). Hence, we cannot exploit the estimate \eqref{kalpha} to deduce the sign of the coefficient $\kappa_\alpha^0(\Sigma;\Lambda)$, that might be also negative. 
\end{Remark}

 
\subsection{The Navier case}
For $\Omega=\{x\in C_\Sigma~:~|x|<1\}$, let us introduce the space
\[
C^2_{N}(\Omega):=\{u\in C^2(\overline\Omega)~:~u|_{\partial\Omega}=0,~u=0\text{ in a neighborhood of $0$}\}.
\]
In the case $\Omega=\{x\in C_\Sigma~:~|x|>1\}$, we define
\[
C^2_{N}(\Omega):=\{u\in C_c^2(\overline\Omega)~:~u|_{\partial\Omega}=0\}.
\]
We point out that, in both cases, the space $C^2_N(\Omega)$ strictly contains $C^2_{ND}(\Omega)$. In particular, if $u\in C^2_{ND}(\Omega)$ then $u=0$ and $\nabla u=0$ on $S^{N-1}\cap\partial\Omega$, whereas if if $u\in C^2_{N}(\Omega)$ then $u=0$ on $S^{N-1}\cap\partial\Omega$ but $\nabla u$ does not necessarily vanish on $S^{N-1}\cap\partial\Omega$. 
The boundary conditions expressed in the above definitions of $C^2_{N}(\Omega)$ correspond to the so-called Navier case for domains $\Omega$ of the form \eqref{Omega}. Finally, as in the previous sections, for every $\Lambda$ nonempty subset of $\Lambda_\Sigma$, we denote by $C^2_{N}(\Omega;\Lambda)$ the set of functions in $C^2_{N}(\Omega)$ satisfying the orthogonality condition \eqref{ortogonale} for all $0<r<1$, respectively, for all $r>1$, in the two cases. We show the following improved Hardy-Rellich inequality:

\begin{Theorem}\label{N-case}
Let $\alpha\in\R$ and let $\Omega$ be as in \eqref{Omega}. Furthermore, let $\Lambda$ be a nonempty subset of $\Lambda_\Sigma$. Then for every $u\in C^2_{N}(\Omega;\Lambda)$ one has
\begin{equation*}
\begin{split}
\int_{\Omega}   |x|^\alpha  |\Delta u|^2\,dx&\ge  \mu_\alpha(\Sigma;\Lambda)\int_{\Omega}|x|^{\alpha-2}|\nabla u|^2\,dx\\&+\kappa_\alpha(\Sigma;\Lambda)\int_{\Omega}|x|^{\alpha-4}\left|\log|x|\right|^{-2}|u|^2\,dx
\end{split}
\end{equation*}
where $\mu_\alpha(\Sigma;\Lambda)$ is given by \eqref{mu-alpha-Sigma} and $\kappa_\alpha(\Sigma;\Lambda)$ as in \eqref{kalpha}. 
Moreover, if $A_0,A_1\in\R$ are such that
\begin{equation*}
\begin{split}
\int_{\Omega}   |x|^\alpha  |\Delta u|^2\,dx&\ge  A_0\int_{\Omega}|x|^{\alpha-2}|\nabla u|^2\,dx +A_1\int_{\Omega}|x|^{\alpha-4}\left|\log|x|\right|^{-2}|u|^2\,dx
\end{split}
\end{equation*}
for every $u\in C^2_{N}(\Omega;\Lambda)$, then $A_0\le\mu_\alpha(\Sigma;\Lambda)$ and $ A_1\le(2\lambda+2\overline\gamma_\alpha-\mu_\alpha(\Sigma;\Lambda))/4$, where $\lambda\in \Lambda$ is such that $\mu_\alpha(\Sigma;\Lambda)=\frac{(\lambda+\gamma_\alpha)^2}{\lambda+\widehat\gamma_\alpha^2}$. 
\end{Theorem}

The proof is the same as Theorem \ref{52} with the only difference that here, in view of the space function chosen, we cannot use \eqref{R1d} anymore and then we just estimate the term involving $w_{tt}$ as follows $\int_Z|w_{tt}|^2\,dt\,{\dsigma}\ge 0$. 

\begin{Remark}
We note that both the optimal constants found in Theorem \ref{N-case} coincide with those found in the Navier-Dirichlet case in Theorem \ref{52}. However, in the Navier case it is not possible to obtain a remainder term involving  $|\nabla u|^2$ with a weight of the form $|x|^{\alpha-2}\left|\log\left|x\right|\right|^{-2}$ as in the Dirichlet case, because the integrability is lost on $\partial\Omega\cap\{|x|=1\}$ when $|\nabla u|\ne 0$. The situation would differ if alternative weights were considered, as done in \cite[Theorem 3.6]{M}.
\end{Remark}

\par\bigskip\noindent
\textbf{Acknowledgments.}
The authors are members of the Gruppo Nazionale per l'Analisi Matematica, la Probabilit\`a e le loro Applicazioni (GNAMPA, Italy) of the Istituto Nazionale di Alta Matematica (INdAM, Italy).\par The research of the first author was carried out within the PRIN 2022 project 2022SLTHCE - Geometric-Analytic Methods for PDEs and Applications GAMPA (CUP E53D23005880006), funded by European Union - Next Generation EU within the PRIN 2022 program (D.D. 104 - 02/02/2022 Ministero dell'Universit\`a e della Ricerca). The research of the second author is partially supported by the PRIN 2022 project 20227HX33Z - Pattern formation in nonlinear phenomena (CUP D53D23005680006), funded by European Union - Next Generation EU within the PRIN 2022 program (D.D. 104 - 02/02/2022 Ministero dell'Universit\`a e della Ricerca, Italy). This manuscript reflects only the authors' views and opinions and the Ministry cannot be considered responsible for them.

\end{document}